\documentclass[11pt]{article}

\usepackage[T1]{fontenc}
\usepackage[utf8]{inputenc}
\usepackage{lmodern}
\usepackage{microtype}
\usepackage[a4paper,margin=1in]{geometry}

\usepackage{amsmath,amssymb,amsfonts,amsthm,mathrsfs}

\usepackage[numbers,sort&compress]{natbib}

\usepackage{graphicx}
\usepackage{algorithm}
\usepackage{algpseudocode}
\usepackage{float}
\usepackage{enumitem}
\usepackage{diagbox}
\usepackage{booktabs}
\usepackage{nicefrac}
\usepackage{array}
\usepackage{makecell}
\usepackage{xcolor}
\usepackage{subcaption}

\usepackage[colorlinks=true,linkcolor=blue,citecolor=blue,urlcolor=blue]{hyperref}

\newtheorem{theorem}{Theorem}[section]
\newtheorem{proposition}[theorem]{Proposition}

\newtheorem{corollary}[theorem]{Corollary}
\newtheorem{assumption}[theorem]{Assumption}

\theoremstyle{definition}

\newtheorem{example}[theorem]{Example}

\theoremstyle{remark}
\newtheorem{remark}[theorem]{Remark}

\numberwithin{equation}{section}

\title{\textbf{Two-grid Penalty Approximation Scheme for Doubly Reflected BSDEs}}

\author{
Wonjae Lee\thanks{Department of Mathematical Sciences, Seoul National University, Seoul 08826, South Korea. \texttt{phwjlee117@gmail.com}}
\and
Hyungbin Park\thanks{Research Institute of Mathematics and Department of Mathematical Sciences, Seoul National University, Seoul 08826, South Korea. \texttt{hyungbin@snu.ac.kr}}
}

\date{\today}

\begin{document}

% Use symbolic footnotes (*, \dagger, \ddagger) for author affiliations
\renewcommand{\thefootnote}{\fnsymbol{footnote}}
\maketitle
\renewcommand{\thefootnote}{\arabic{footnote}}
\setcounter{footnote}{0}

\begin{abstract}
We propose and analyze a penalty-based approximation scheme for decoupled Markovian backward stochastic differential equations (BSDEs) with double obstacles. In contrast to the single-obstacle case, in the double-obstacle setting, the barrier approximation error amplified by the penalty parameter cannot be eliminated using a simple shift argument. To address this challenge, we introduce a two-grid method that combines a fine-grid simulation of the forward SDE with a coarse-grid Euler discretization of the penalized backward equation. Under suitable structural assumptions on the barriers, we derive a penalization error of order $O(\lambda^{-1})$ with respect to the penalty parameter $\lambda$, and establish error bounds for the fully discrete scheme. When the driver is independent of the integrand process, the method achieves the optimal order $O(\Delta t^{1/2})$ for backward grid size $\Delta t$ under an appropriate parameter coupling. Numerical experiments on the game put option illustrate and confirm the theoretical findings.
\end{abstract}


\section{Introduction}

Backward stochastic differential equations (BSDEs) with reflection arise naturally in stochastic control, mathematical finance, and game theory. Reflected BSDEs with a single barrier provide probabilistic representations of optimal stopping problems and obstacle partial differential equations. Doubly reflected BSDEs (DRBSDEs), which are subject to two barriers, are closely connected to Dynkin games, variational inequalities with two obstacles, and the valuation of game-type contingent claims such as callable, putable, and convertible securities; see, for example, \cite{BCJR:09,CK:96}.

This study investigates an approximation method for DRBSDEs. We consider a decoupled Markovian DRBSDE
\begin{equation}\label{eq:intro_drbsde}
\left\{\quad
\begin{aligned}
&dX_t = b(t,X_t)\,dt + \sigma(t,X_t)\,dW_t\,,\\
&Y_t = g(X_T) + \int_t^T f(s,X_s,Y_s,Z_s)\,ds \\
&\qquad\qquad+ (A_T-A_t) - (K_T-K_t) - \int_t^T Z_s\,dW_s\,,\;\;0\le t\le T
\end{aligned}
\right.
\end{equation}
subject to the double-obstacle constraint
$$p_b(t,X_t)\le Y_t \le p_w(t,X_t)\,,\; t\in [0,T]\,.$$
A classical approach to approximating the reflected BSDEs is the penalization method, in which the reflection terms are replaced by large penalty terms whenever the solution exits the admissible region. In this setting, the DRBSDE \eqref{eq:intro_drbsde} is approximated by the penalized BSDE
\begin{equation}\label{eq:intro_penalized}
Y_t^\lambda
= g(X_T) + \int_t^T f_\lambda(s,X_s,Y_s^\lambda,Z_s^\lambda)\,ds
- \int_t^T Z_s^\lambda\,dW_s,
\end{equation}
  where
\[
f_\lambda(t,x,y,z)
:= f(t,x,y,z) + \lambda\,(y-p_b(t,x))^- - \lambda\,(p_w(t,x)-y)^-,
\quad \lambda>0\,.
\]

In practice, the penalized BSDE is solved numerically using time discretization. In most existing studies, the discretization of the penalized equation \eqref{eq:intro_penalized} is based on a grid given by a single partition of \([0,T]\). However, in such schemes, the explicit error is either suboptimal, or controlled through a shifting trick that is unavailable for the double-obstacle case. 
To reduce the error to standard rates, we instead employ two grids given by nested partitions
\[
\pi=\{t_i\}_{i=0}^n \subseteq \tilde{\pi}=\{\tilde t_j\}_{j=0}^{mn}
\]
of \([0,T]\), with coarse and fine step sizes
\[
\Delta t := T/n,
\qquad
\tilde{\Delta t} := T/(mn),
\]
respectively. More precisely, we discretize the penalized DRBSDE \eqref{eq:intro_penalized} on the coarse grid \(\pi\), while the forward process \(X\) is simulated on the finer grid \(\tilde{\pi}\) by the Euler--Maruyama scheme and then evaluated at the coarse-grid times.
With $\Delta W_{i+1}=W_{t_{i+1}}-W_{t_i}$,  the scheme is given by
\begin{equation*}
    \begin{cases}
        \mathcal{Y}^{\lambda,\pi}_{t_{i+1}}
:=
\mathcal{Y}^{\lambda,\pi}_{t_i}
-
f_\lambda\left(t_i,X^{\tilde{\pi}}_{t_i},\mathcal{Y}^{\lambda,\pi}_{t_i},\mathcal{Z}^{\lambda,\pi}_{t_i}\right)\Delta t
+
\mathcal{Z}^{\lambda,\pi}_{t_i}\Delta W_{i+1},\\
\mathcal Y^{\lambda,\pi}_{t_n}=g(X^{\tilde\pi}_{t_n}),
    \end{cases}
\end{equation*}
for $i=n-1,\dots,0$.

Several challenges naturally arise in implementing this procedure: (i) understanding how the penalty parameter \(\lambda\) interacts with the time discretization, (ii) identifying choices of \(\lambda\) and \(\tilde{\Delta t}\) that yield an efficient implementation, and (iii) determining the convergence rate achieved by the resulting approximation scheme. This study addresses these questions.

\medskip
\noindent\textbf{Related Works}
Although studies specifically addressing explicit error bounds for penalty-based numerical schemes for (doubly) reflected BSDEs are scarce, we mention the following works that are related to the present study.
\begin{itemize}

\item In \cite{BC:08}, the authors studied a non-penalization-based discrete-time approximation scheme for reflected BSDEs in a Markovian setting, with obstacle of the form $h(X)$ for a forward diffusion $X$. They established a convergence rate of order $O(\Delta t^{1/4})$ for the discretely reflected problem, and obtained the improved rate $O(\Delta t^{1/2})$ under the additional smoothness assumption $h\in C_b^2$.

\item The corresponding double-barrier case was studied in \cite{C:09}, where an Euler-type approximation for doubly reflected BSDEs was introduced. Under varying regularity assumptions, squared-error bounds of order $O(\Delta t^{q})$ were obtained for suitable $q\in\{1/6,\,1/3,\,1/2\}$.

\item It was shown in \cite{X:11}, which proposes a binomial-tree-based discretization scheme for BSDEs with two continuous barriers, that the penalized solutions converge to the solution of the original DRBSDE, with penalization error of order $O(\lambda^{-1/2})$.

\item In \cite{PWW:24}, the authors proposed the Deep Penalty Method, which relies on a penalty discretization scheme for reflected BSDEs similar to ours, with drivers of the form \(f(t,x)-ry\), and its penalized version
\[
f(t,x)-ry-\lambda(y-p(t,x))^-, \qquad \lambda>0.
\]
Under additional assumptions on the $C^{1,2}$ barrier function $p$ and by an application of the comparison principle, an improved $O(\lambda^{-1})$ bound on the penalization error is obtained. In addition, the shift $Y_\cdot-p(\cdot,X_\cdot)$ is used to remove the forward-simulation error amplified by $\lambda$. Furthermore, an explicit absolute error bound is derived, and the choice \(\lambda\asymp \Delta t^{-1/2}\) yields the standard $O(\Delta t^{1/2})$ convergence rate.

\end{itemize}

\medskip
\noindent\textbf{Contributions.}
The main contributions of this study are summarized as follows.

\begin{enumerate}
\item \textbf{Improved penalization error under structural assumptions.}
Building on a partial differential equation (PDE) comparison/viscosity argument in the spirit of \cite{PWW:24} and under additional assumptions tailored to financial barriers, we improve the pure penalization bound and obtain
\[
|Y_t-Y_t^\lambda|\le \frac{C}{\lambda}
\]
uniformly in \(t\), improving the \(\lambda^{-1/2}\) bound in \cite{X:11}. 
This sharper penalization rate is essential for balancing the total approximation error in Corollary \ref{cor:final}.

\item \textbf{Two-grid forward simulation for double barriers.} We propose a two-grid discretization to mitigate the error from simulation of the forward process amplified by $\lambda$. Our method restores the standard \(O(\Delta t^{1/2})\) convergence behavior of the value approximation with respect to the backward time step, without requiring the backward scheme to be run on a fine grid. For the two barrier case, no transformation analogous to the shift $Y_\cdot-p(\cdot,X_\cdot)$ in \cite{PWW:24} can be deployed. Instead, our method simulates the forward SDE on a finer grid \(\tilde\pi\) and discretizes the penalized backward equation on a coarser grid \(\pi\subset\tilde\pi\). This separation is essential in the double-barrier case because the error incurred in evaluating the barriers \(p_b(t,X_t)\) and \(p_w(t,X_t)\) along the forward approximation appears multiplied by the penalty parameter \(\lambda\). By refining only the forward simulation, the scheme controls this amplified error at relatively low computational cost, while preserving a cost-effective coarse-grid backward recursion.

\item \textbf{Explicit error rates with parameter coupling.}
We derive quantitative error bounds for the full penalization--discretization approximation
\((X^{\tilde\pi},\mathcal{Y}^{\lambda,\pi},\mathcal{Z}^{\lambda,\pi})\), with explicit dependence on the coarse and fine step sizes
\(\Delta t\) and \(\tilde{\Delta t}\), and on the penalty parameter \(\lambda\).
In the general Lipschitz case, allowing \(Z\)-dependence in the driver, we prove the mean-square estimate
\[
\Big(\max_{0\le i\le n-1}\mathbb{E}|Y_{t_i}-\mathcal{Y}^{\lambda,\pi}_{t_i}|^2\Big)^{1/2}
\le
C\Big(\lambda \Delta t^{1/2}+\lambda^{-1}\Big),
\]
under the natural terminal approximation, so that the balancing choice
\(\lambda\asymp \Delta t^{-1/4}\) yields the rate \(O(\Delta t^{1/4})\) for the value process.
When the driver is independent of \(Z\), we further obtain the absolute-error bound
\[
\max_{0\le i\le n-1}\mathbb{E}|Y_{t_i}-\mathcal{Y}^{\lambda,\pi}_{t_i}|
\le
C\Big(\Delta t^{1/2}+\lambda \tilde{\Delta t}^{1/2}+\lambda\Delta t+\lambda^{-1}\Big),
\]
which exhibits the two-grid coupling explicitly. In particular, choosing
\(\lambda\asymp \Delta t^{-1/2}\) together with
\(\tilde{\Delta t}=O(\Delta t/\lambda^2)\) recovers the target rate
\[
\max_{0\le i\le n-1}\mathbb{E}|Y_{t_i}-Y^{\lambda,\pi}_{t_i}| = O(\Delta t^{1/2}).
\]
These results extend the single-obstacle analysis of \cite{PWW:24} to the double-barrier setting. In contrast to \cite{PWW:24}, which treats drivers linear in \(Y\) and independent of \(Z\), our framework allows drivers that are Lipschitz in both \(Y\) and \(Z\).

\item \textbf{Nonsmooth payoffs and barriers via multivariate It\^o--Tanaka.}
Financial barriers such as basket calls and puts are typically not globally \(C^{1,2}\), as assumed in \cite{PWW:24}. We allow barrier functions that are \(C^{1,2}\) away from a finite union of \(C^2\) hypersurfaces. The analysis uses Tanaka-type arguments together with the change-of-variable formula with local time on surfaces (Peskir, 2007), which allows us to handle kink-type singularities while still deriving quantitative bounds compatible with penalization and discretization.

\item \textbf{Numerical validation and finite-sample behavior.}
We complement the theoretical analysis with numerical experiments for a one-dimensional game put under the Black--Scholes model. The grid-refinement experiment, performed under the theoretically motivated coupling \(\lambda\asymp n^{1/2}\), exhibits relative errors consistent with an \(n^{-1/2}\) decay. A second penalty sweep at fixed \(n\) indicates that over the tested range, the error continues to decrease as \(\lambda\) increases, indicating that the asymptotic balancing effect is not present in our example.
\end{enumerate}

\medskip
\noindent\textbf{Organization.}
The remainder of the paper is structured as follows. Section~2 introduces the DRBSDE framework and basic assumptions, together with existence and uniqueness results. Section~3 establishes the penalization error and the improved \(O(\lambda^{-1})\) bound. Section~4 presents the two-grid numerical scheme and derives the explicit discretization error bounds, along with parameter couplings that yield the \(O(\Delta t^{1/2})\) rate. Section~5 reports the numerical experiments for a game put example under the Black--Scholes model. Section~6 presents the conclusion.

\section{Doubly reflected BSDEs: setting and preliminaries}
\label{sec:prelim}
We now describe the framework and introduce Markovian doubly reflected BSDEs under suitable assumptions. Let us fix a time horizon $T>0$ and consider a filtered probability space 
\[
(\Omega, \mathcal{F}_T, \mathbb{F} = (\mathcal{F}_t)_{t \in [0,T]}, \mathbb{P}).
\]
 Let $W = (W_t)_{t \in [0,T]}$ be the standard $\mathbb{R}^d$-valued Brownian motion defined on the above probability space. 
For $m\in\mathbb{N}$ and $p\ge2$, we use the following notations.

\begin{itemize}
    \item $L^{p}_{\mathcal{F}_T}(\mathbb{R}^m)$ is the space of all $\mathcal{F}_T$-measurable $\mathbb{R}^m$-valued random variables $X$ such that 
    $$\|X\|^p:=\mathbb{E}\left[|X|^p\right]<\infty.$$
    
    \item $\mathbb{H}^p(\mathbb{R}^m)$ is the space of all $\mathbb{R}^m$-valued progressively measurable processes $\phi$ such that
    $$\|\phi\|^p_{\mathbb{H}^p}:=\mathbb{E}\left[\left(\int^{T}_{0}|\phi_t|^2\,dt\right)^\frac{p}{2}\right]<\infty.$$

    \item $\mathbb{S}^p(\mathbb{R}^m)$ is the space of all adapted continuous processes $\phi$ such that
    $$\|\phi\|^p_{\mathbb{S}^p}:=\mathbb{E}\left[\sup_{t\in[0,T]}|\phi_t|^p \right]<\infty.$$

    \item $\mathcal{A}$ is the space of real-valued non-decreasing continuous predictable processes $A$ with $A_0=0$. We also define
    $$\mathcal{A}^{2}:=\{A\in\mathcal{A}\mid\mathbb{E}\left[A^2_T\right]<\infty\}.$$

\end{itemize}

We now proceed to define doubly reflected BSDEs and state our assumptions. For $X_0\in L^2(\Omega)$, we consider BSDEs of the form

\begin{equation}
\label{BSDE}
\begin{cases}
    \text{(i)}\ X_t=X_0+\int^t_0 b(s,X_{s})\,ds + \int^t_0 \sigma(s,X_{s})^\top\,dW_s,\quad t\in[0,T],\\
    \text{(ii)}\ Y_t = g(X_T)+\int^T_t f\left(s,X_{s},Y_{s},Z_s\right)\,ds +(A_T-A_t)\\
    \ \ \ \ \ \ \ \ \ \  -(K_T-K_t) -\int^T_t (Z_s)^\top\,dW_s, \quad t\in[0,T],\\
    \text{(iii)}\ \forall t\in[0,T],\ \xi_t:=p_b(t,X_t) \le Y_t \le \zeta_t:=p_w(t,X_t)\ a.s.\\
    \text{(iv)}\ \int^T_0(Y_{t}-\xi_{t})dA_t=0\ a.s.\ \text{and} \int^T_0(\zeta_{t}-Y_{t})dK_t=0\ a.s.
\end{cases}
\end{equation}
under the following assumptions on the coefficients.

\begin{assumption}
\label{assumptions}
\leavevmode
    \begin{enumerate}[label=(\roman*), leftmargin=*, align=left]
        \item The maps $b:[0,T]\times\mathbb{R}^d\to \mathbb{R}^d$, $\sigma:[0,T]\times\mathbb{R}^d\to\mathbb{R}^{d\times d}$ are $\frac{1}{2}$-H\"older continuous in the first argument. That is, there exist constants $K_{b,\frac{1}{2}},K_{\sigma,\frac{1}{2}},$ such that
        \begin{align*}
            |b(t,x)-b(t',x)|\le K_{b,\frac{1}{2}}|t-t'|^\frac{1}{2},\quad
            |\sigma(t,x)-\sigma(t',x)|\le K_{\sigma,\frac{1}{2}}|t-t'|^\frac{1}{2}
        \end{align*}
        for all $t,t'\in[0,T]$ and $x\in\mathbb{R}^d$. Also, $b$ and $\sigma$ are Lipschitz continuous in the second argument, i.e., there exist constants $K_b,K_\sigma>0$ such that
        \begin{align*}
            |b(t,x)-b(t,y)|\le K_b|x-y|,\quad
            |\sigma(t,x)-\sigma(t,y)|\le K_\sigma|x-y|
        \end{align*}
        for all $t\in[0,T]$  and $x,y\in\mathbb{R}^d$.

        \item The function $g:\mathbb{R}^d\to\mathbb{R}$ is Lipschitz continuous, i.e., there exists a constant $K_g$ such that
        $$|g(x)-g(y)|\le K_g|x-y|$$
        for all $x,y\in\mathbb{R}^d$
        
        \item There exist constants $K_x, K_y,K_z>0$ such that the function $f:[0,T]\times\mathbb{R}^d\times\mathbb{R}\times\mathbb{R}^d\to \mathbb{R}$ satisfies
        $$|f(t,x,y,z)-f(t,x',y',z')|\le K_x|x-x'|+K_y|y-y'|+K_z|z-z'|$$
        for all $(t,x,y,z),(t,x',y',z')\in[0,T]\times\mathbb{R}^d\times\mathbb{R}\times\mathbb{R}^d$. Also, 
        $$\|f(0,\cdot,0,0)\|_\infty<\infty$$
        and there exists a constant $K_{f,\frac{1}{2}}$ such that 
        $$|f(t,x,y,z)-f(t',x,y,z)|\le K_{f,\frac{1}{2}}|t-t'|^\frac{1}{2}$$
        for all $(t,x,y,z),(t',x',y',z')\in [0,T]\times\mathbb{R}^d\times\mathbb{R}\times\mathbb{R}^d$. 
        
        \item The functions $p_w, p_b:[0,T]\times\mathbb{R}^d\to\mathbb{R}$ are $\frac{1}{2}$-H\"older continuous in the first argument with H\"older constants $K_{p_w,\frac{1}{2}},K_{p_b,\frac{1}{2}}$, respectively, and Lipschitz-continuous in the second arguments with Lipschitz constants $K_{p_w},K_{p_b}$, respectively.

        \item We have $p_b(t,x)< p_w(t,x)$ and
       $p_b(T,x)\le g(x)\le p_w(T,x)$ for all $t\in[0,T)$ and $x\in\mathbb{R}^d$.
    
    \end{enumerate}
\end{assumption}

With Assumption~\ref{assumptions}, we can establish the following theorem of uniqueness and existence of solution to \eqref{BSDE} as a straightforward corollary to Theorem 3.7 of \cite{HH:05}.

\begin{theorem}
    Under Assumption~\ref{assumptions}, there exists a unique solution $(X,Y,Z,A,K)\in\mathbb{S}^2(\mathbb{R}^d)\times\mathbb{S}^2(\mathbb{R})\times\mathbb{H}^2(\mathbb{R}^d)\times\mathcal{A}^2\times \mathcal{A}^2$ to the BSDE \eqref{BSDE}.
\end{theorem}

Furthermore, it is well known that \eqref{BSDE} has a connection to the free-boundary PDE for $V:[0,T]\times\mathbb{R}^d\to\mathbb{R}$ given by
\begin{equation}
\label{fbpde}
    \begin{cases}
        \max\left(\min(-\mathcal{L}V-f(\cdot,\cdot,V,\sigma\nabla_xV),V-p_b),V-p_w\right)=0\text{ on } [0,T)\times\mathbb{R}^d\\
        p_b(t,x)\le V(t,x)\le p_w(t,x),\\
        V(T,x)=g(x),
    \end{cases}
\end{equation}
where
\begin{equation*}
\begin{aligned}
    \mathcal{L}V(t,x):=\partial_tV(t,x)+&b(t,x)^\top\nabla_xV(t,x)+\frac{1}{2}\sigma(t,x)^\top\nabla_x^2 V(t,x)\sigma(t,x).
\end{aligned}
\end{equation*}
Specifically, we have the following well-known result, as can be checked in \cite{DQS:16}. 
\begin{theorem}
    Suppose Assumption~\ref{assumptions} holds. Let $(X,Y,Z,A,K)$ be a solution to \eqref{BSDE}. The PDE \eqref{fbpde} has a unique viscosity solution $V$ and we have
    $$Y_t=V(t,X_t)$$
    almost surely for all $t\in[0,T]$.
\end{theorem}
For our method, we consider the following penalized BSDE which approximates \eqref{BSDE} as $\lambda\to\infty$:
\begin{equation}
\label{pBSDE}
\begin{cases}
    \text{(i)}\ X_t=X_0+\int^t_0 b(s,X_{s})\,ds + \int^t_0 \sigma(s,X_{s})^\top\,dW_s,\quad t\in[0,T],\\
    \text{(ii)}\ Y^\lambda_t = g(X_T)+\int^T_t f_\lambda\left(s,X_{s},Y^\lambda_s,Z^\lambda_s\right)\,ds-\int^T_t (Z^\lambda_s)^\top\,dW_s, \quad t\in[0,T].
\end{cases}
\end{equation}
Here, $\lambda>0$ and $f_\lambda:[0,T]\times\mathbb{R}^d\times\mathbb{R}\to \mathbb{R}$ are defined by
$$f_\lambda(t,x,y,z):=f(t,x,y,z)+\lambda\left(y-p_b(t,x)\right)^{-}-\lambda\left(p_w(t,x)-y\right)^{-}.$$
We aim to calculate the error bound of the discretization scheme for the penalized BSDE \eqref{pBSDE}. Again, if Assumption~\ref{assumptions} holds, we have the following uniqueness and existence result (See, e.g., \cite{MY:99}).
\begin{theorem}
        Under Assumption~\ref{assumptions}, there exists a unique solution $(X,Y^\lambda,Z^\lambda)\in\mathbb{S}^2(\mathbb{R}^d)\times\mathbb{S}^2(\mathbb{R})\times\mathbb{H}^2(\mathbb{R}^d)$ to the BSDE \eqref{pBSDE}.
\end{theorem}

Now, we discuss the PDE corresponding to \eqref{pBSDE}. Specifically, for $V:[0,T]\times\mathbb{R}^d\to\mathbb{R}$, we consider the PDE
\begin{equation}
\label{pide}
\begin{aligned}    \mathcal{L}V^\lambda(t,x)+f_\lambda\left(t,x,V^\lambda(t,x),\sigma(t,x)\nabla_xV^\lambda(t,x)\right)&=0, \quad (t,x)\in[0,T)\times\mathbb{R}^d\\
    V^\lambda(T,x)&=g(x),\quad x\in\mathbb{R}^d.
\end{aligned}
\end{equation}
It is well known that, under suitable conditions, the solution to the BSDE \eqref{pBSDE} can be expressed in terms of the solution $V$ to the PDE \eqref{pide}. Specifically, we have the following result (see, for example, \cite{MZ:02}):
\begin{assumption}
\label{assumptionpde}
\leavevmode
    \begin{enumerate}[label=(\roman*), leftmargin=*, align=left]
        \item Either $f(t,x,y,z)=f(t,x,y,0)$ for all $(t,x,y,z)\in[0,T]\times\mathbb{R}^d\times\mathbb{R}\times\mathbb{R}^d$, or $\sigma$ is uniformly elliptic on $[0,T]\times \mathbb{R}^d$ .
        \item $b$ and $\sigma$ are $C^1_b$ in the second argument.
    \end{enumerate}
\end{assumption}
\begin{theorem}\label{thmpde}
Under Assumptions~\ref{assumptions} and \ref{assumptionpde}, the penalized PDE \eqref{pide}
admits a unique viscosity solution $V^\lambda\in C([0,T]\times\mathbb{R}^d,R)$ with at most linear growth. Let $X$ be the solution of the forward SDE in \eqref{pBSDE}.
Then, for the unique solution $(X,Y^\lambda,Z^\lambda)$ of \eqref{pBSDE}, we have
\[
Y_t^\lambda = V^\lambda(t,X_t), \qquad t\in[0,T], \ \text{a.s.}
\]
If, in addition, $\sigma$ is uniformly elliptic on $[0,T]\times\mathbb{R}^d$, then
\[
V^\lambda \in C_b^{0,1}([0,T)\times\mathbb{R}^d,\mathbb{R})
\]
and
\[
Z_t^\lambda = \sigma(t,X_t)\nabla_x V^\lambda(t,X_t), \qquad t\in[0,T], \ \text{a.s.}
\]
\end{theorem}

\section{Penalty error}
\label{sec:peab}
To analyze the error from penalization, we extend Assumption 3 of \cite{PWW:24} to double barriers given by payoff functions that may not be $C^{1,2}$ on a union of regular $C^2$ hypersurfaces.
\begin{assumption}\label{assumption2}
There exist $C^2$ hypersurfaces $\Gamma_1,\dots,\Gamma_q \subset \mathbb{R}^d$ such that
$\Gamma_k=\{x\in\mathbb{R}^d:\varphi_k(x)=0\}$ for some $\varphi_k\in C^2(\mathbb{R}^d)$ with
$\inf_{x\in\Gamma_k}|\nabla\varphi_k(x)|>0$, satisfying the following conditions. We
set $\Gamma:=\bigcup_{k=1}^q \Gamma_k$ and $U:=\mathbb{R}^d\setminus \Gamma$.

\begin{enumerate}[label=(\roman*),leftmargin=*]
\item We have
$p_b,p_w\in C^{1,2}\big([0,T)\times U\big)$ and $\mathbb{P}(X_0\in \Gamma)=0$.
Moreover, for each $k$ and each $p\in\{p_b,p_w\}$, the one-sided normal derivatives
$\nabla_{n_k}p(t,x\pm)$ exist for $(t,x)\in[0,T)\times \Gamma_k$, where
$n_k(x):=\nabla\varphi_k(x)/|\nabla\varphi_k(x)|$.

\item
There exists $J>0$ such that for each $k$ and each $p\in\{p_b,p_w\}$,
$$\sup_{(t,x)\in[0,T)\times\Gamma_k}\Big|\big[\nabla_{n_k}p\big](t,x)\Big|\le J ,$$
where
\[
\sup_{(t,x)\in[0,T)\times\Gamma_k}\Big|\big[\nabla_{n_k}p\big](t,x)\Big|
\;:=\;\sup_{(t,x)\in[0,T)\times\Gamma_k}\Big|\nabla_{n_k}p(t,x+)-\nabla_{n_k}p(t,x-)\Big|
.
\]

\item 
There exist $P>0$ and $r\ge 0$ such that for each $p\in\{p_b,p_w\}$,
\[
|\mathcal{L}p(t,x)|\le P(1+|x|^r),\qquad (t,x)\in[0,T)\times U .
\]
\item We have $D_b>-\infty$ and $D_w<\infty$, where
    \begin{equation*}
        D_b:=\inf_{(t,x)\in[0,T)\times (\mathbb{R}^d\setminus\Gamma)} \mathcal{L}p_b(t,x)+f\left(t,x,p_b(t,x),\sigma(t,x)\nabla_xp_b(t,x)\right)
    \end{equation*}
    and
    \begin{equation*}
    D_w:=\sup_{(t,x)\in[0,T)\times (\mathbb{R}^d\setminus\Gamma)} \mathcal{L}p_w(t,x)+f\left(t,x,p_w(t,x),\sigma(t,x)\nabla_xp_w(t,x)\right)\,.
    \end{equation*}
\end{enumerate}
\end{assumption}

\begin{example}
    
Consider a put option in which the holder may exercise early, while the writer may also cancel early by paying an additional penalty. The underlying consists of $d$ assets. Under a risk-neutral measure, the asset dynamics are given by the Black--Scholes SDE
\begin{equation}
    dX_t^i = r X_t^i\,dt + \sigma X_t^i\,dW_t^i.
\end{equation}

The holder's exercise payoff $p_b:[0,T]\times\mathbb{R}^d\to\mathbb R$ is the arithmetic basket put
\[
    p_b(t,x)=\min\left(\Bigl(K-\frac{1}{d}\sum_{i=1}^d x^{i}\Bigr)^+, K\right)
   ,
\]
and the writer's cancellation payoff $p_w:[0,T]\times\mathbb{R}^d\to\mathbb R$ is defined by adding a time-dependent penalty
\begin{equation}\label{eq:israeli_payoffs}
    p_w(t,x) := p_b(t,x) + \delta (T-t),
\end{equation}
so that $p_b(t,x)\le p_w(t,x)$ for all $(t,x)$ and $p_w(T,x)=p_b(T,x)$.

By the standard Dynkin-game representation and its characterization via doubly reflected BSDEs
(see, e.g., \cite{CK:96}),
the undiscounted value process $Y$ is characterized as the $Y$-component of the DRBSDE \eqref{BSDE} with coefficients
\[
    b(s,x)=r\,x,\qquad
    \sigma(s,x)=\sigma\,\mathrm{diag}(x^1,\dots,x^d),\qquad
    f(s,x,y,z)=-r\,y
\]
with terminal payoff $g(x)=p_b(T,x)=p_w(T,x)$ and obstacles $\xi_t=p_b(t,X_t)$, $\zeta_t=p_w(t,X_t)$.

Then, it is straightforward to check that Assumptions~\ref{assumptions} and \ref{assumptionpde} hold. Furthermore, with $\Gamma_1=\{x\in\mathbb{R}^d\mid K-\frac{1}{d}\sum^d_{i=1}x^i=0\}$ and $\Gamma_2=\{x\in\mathbb{R}^d\mid \sum^d_{i=1}x^i=0\}$ as the hypersurfaces of singularity, it is clear that Assumption~\ref{assumption2} holds as well.

\end{example}

The following proposition is analogous to Proposition 1 in \cite{PWW:24}.
\begin{proposition}
\label{apb1}
    Let $\lambda>K_y$. Suppose Assumptions \ref{assumptions},\ref{assumptionpde} and \ref{assumption2} hold. Let us denote by $V^\lambda$ the solution to the PDE \eqref{pide}. Then,
    $$(\lambda-K_y)(V^\lambda(t,x)-p_b(t,x))^-\le D_b^-$$
    and 
    $$(\lambda-K_y)(p_w(t,x)-V^\lambda(t,x))^-\le D_w^+$$
    for all $(t,x)\in[0,T]\times \mathbb{R}^d$
\end{proposition}

\begin{proof}
We prove the first inequality. The second one is proved analogously. We begin by letting $Q(x):=(1+|x|^2)^m$, and choose \(m, a>0\) large enough to satisfy
\begin{equation}\label{eq:Qineq}
\mathcal{L}Q(x)-aQ(x)+K_y|Q(x)|+K_z|\sigma(t,x)\nabla Q(x)|\le 0
\end{equation}
for all \((t,x)\in[0,T]\times\mathbb{R}^d\). Also, for \(\varepsilon>0\), we define
\[
\ell(t,x):=p_b(t,x)-V^\lambda(t,x)-\varepsilon e^{-at}Q(x).
\]
Since \(p_b\) and \(V^\lambda\) have at most linear growth whereas \(Q\) has superlinear growth, there exists a compact set
\[
K=[0,T]\times[-M,M]^d
\]
such that \(\ell<0\) on \(([0,T]\times\mathbb{R}^d)\setminus K\).

Note that if \(\ell\le 0\) on \([0,T]\times\mathbb{R}^d\), then
\[
(V^\lambda(t,x)-p_b(t,x))^-\le \varepsilon e^{-at}Q(x),
\]
and letting \(\varepsilon\downarrow 0\) yields the conclusion. We may therefore assume that \(\ell\) is positive somewhere, and hence attains a maximum on $K$.

\noindent Let us now consider the following two cases.

\medskip
\noindent
\textbf{Case 1.}
Let us assume that \(\ell\) attains a positive maximum at some
\[
(t_0,x_0)\in[0,T)\times(\mathbb{R}^d\setminus\Gamma).
\]
Then, since \(x_0\notin\Gamma\), the function \(p_b\) is \(C^{1,2}\) in a neighborhood of \((t_0,x_0)\). Now we set
\[
\phi(t,x):=p_b(t,x)-\varepsilon e^{-at}Q(x)-\ell(t_0,x_0).
\]
Then
\[
V^\lambda(t_0,x_0)=\phi(t_0,x_0),
\]
and the maximality of \(\ell\) at \((t_0,x_0)\) gives
\[
V^\lambda(t,x)\ge \phi(t,x)
\]
for \((t,x)\) near \((t_0,x_0)\). Thus \(V^\lambda-\phi\) has a local minimum \(0\) at \((t_0,x_0)\), and the viscosity supersolution property of \(V^\lambda\) yields
\begin{equation}\label{eq:viscsuper}
\mathcal{L}\phi(t_0,x_0)+f_\lambda\bigl(t_0,x_0,\phi(t_0,x_0),\sigma(t_0,x_0)\nabla\phi(t_0,x_0)\bigr)\ge 0.
\end{equation}

From
\[
p_b(t_0,x_0)-\phi(t_0,x_0)
=
\ell(t_0,x_0)+\varepsilon e^{-at_0}Q(x_0)>0,
\]
we obtain
\[
(\phi(t_0,x_0)-p_b(t_0,x_0))^-
=
\ell(t_0,x_0)+\varepsilon e^{-at_0}Q(x_0).
\]
Since \(\phi(t_0,x_0)<p_b(t_0,x_0)\le p_w(t_0,x_0)\), we also have
\[
(p_w(t_0,x_0)-\phi(t_0,x_0))^-=0.
\]
Substituting these relations into \eqref{eq:viscsuper}, we obtain
\begin{equation}
\label{eq:prop1}
\begin{aligned}
0
&\le
\mathcal{L}\phi(t_0,x_0)
+
f\bigl(t_0,x_0,\phi(t_0,x_0),\sigma(t_0,x_0)\nabla\phi(t_0,x_0)\bigr)
\\&\qquad+
\lambda\bigl(\phi(t_0,x_0)-p_b(t_0,x_0)\bigr)^-\\
&=
\mathcal{L}p_b(t_0,x_0)
-\varepsilon e^{-at_0}\bigl(\mathcal{L}Q(x_0)-aQ(x_0)\bigr)
\\&\qquad+
f\bigl(t_0,x_0,\phi(t_0,x_0),\sigma(t_0,x_0)\nabla\phi(t_0,x_0)\bigr)\\
&\qquad\qquad
+\lambda\bigl(\ell(t_0,x_0)+\varepsilon e^{-at_0}Q(x_0)\bigr).
\end{aligned}
\end{equation}
Furthermore, the Lipschitz continuity of \(f\) in \((y,z)\) gives
\begin{align*}
&f\bigl(t_0,x_0,\phi(t_0,x_0),\sigma(t_0,x_0)\nabla\phi(t_0,x_0)\bigr)\\
&\ge
f\bigl(t_0,x_0,p_b(t_0,x_0),\sigma(t_0,x_0)\nabla p_b(t_0,x_0)\bigr)
\\&\qquad-
K_y\bigl(\ell(t_0,x_0)+\varepsilon e^{-at_0}Q(x_0)\bigr)
-
\varepsilon e^{-at_0}K_z|\sigma(t_0,x_0)\nabla Q(x_0)|.
\end{align*}
Substituting this into \eqref{eq:prop1}, we arrive at
\begin{align*}
&(\lambda-K_y)\ell(t_0,x_0)
\\\le\;&
-\Bigl(
\mathcal{L}p_b(t_0,x_0)
+
f\bigl(t_0,x_0,p_b(t_0,x_0),\sigma(t_0,x_0)\nabla p_b(t_0,x_0)\bigr)
\Bigr)\\
&\quad
+\varepsilon e^{-at_0}
\Bigl(
\mathcal{L}Q(x_0)-aQ(x_0)+K_y|Q(x_0)|+K_z|\sigma(t_0,x_0)\nabla Q(x_0)|
\Bigr).
\end{align*}
By \eqref{eq:Qineq} and the definition of \(D_b\), this yields
\[
(\lambda-K_y)\ell(t_0,x_0)\le D_b^-.
\]
Since \(\ell(t,x)\le \ell(t_0,x_0)\) for all \((t,x)\in[0,T]\times\mathbb{R}^d\), we obtain
\[
(\lambda-K_y)\Bigl((V^\lambda(t,x)-p_b(t,x))^--\varepsilon e^{-at}Q(x)\Bigr)\le D_b^-,
\qquad \forall (t,x)\in[0,T]\times\mathbb{R}^d.
\]
By letting \(\varepsilon\downarrow 0\), we conclude that
\[
(\lambda-K_y)(V^\lambda(t,x)-p_b(t,x))^-\le D_b^-.
\]

\medskip
\noindent
\textbf{Case 2.}
We now assume that every maximizer of \(\ell\) lies on \([0,T]\times(\Gamma\cap K)\).

Since \(\Gamma\) has empty interior and \(\ell\) is continuous, there exists \((\bar t,\bar x)\in K\) with \(\bar x\notin\Gamma\) such that
\[
\ell(\bar t,\bar x)\ge \sup_{(t,x)\in K}\ell(t,x)-\eta.
\]
for any $\eta>0$. Applying Ekeland's variational principle \citep{Ekeland1974} to the continuous function \(-\ell\) on the complete metric space \(K\), we conclude that for every \(\delta>0\), there exists \((t_\eta,x_\eta)\in K\) such that
\[
\ell(t_\eta,x_\eta)\ge \ell(\bar t,\bar x),
\qquad
|(t_\eta,x_\eta)-(\bar t,\bar x)|\le \delta,
\]
and such that
\[
\ell_\eta(t,x):=\ell(t,x)-\frac{\eta}{\delta}\bigl(|t-t_\eta|^2+|x-x_\eta|^2\bigr)
\]
attains its maximum over \(K\) at \((t_\eta,x_\eta)\).

Choosing \(\delta<\mathrm{dist}(\bar x,\Gamma)\), we have \(x_\eta\notin\Gamma\), and hence \(p_b\) is \(C^{1,2}\) in a neighborhood of \((t_\eta,x_\eta)\). Now, let
\[
\rho_\eta(t,x):=\frac{\eta}{\delta}\bigl(|t-t_\eta|^2+|x-x_\eta|^2\bigr),
\]
and
\[
\phi_\eta(t,x):=
p_b(t,x)-\varepsilon e^{-at}Q(x)-\rho_\eta(t,x)-\ell_\eta(t_\eta,x_\eta),
\]
then since
\[
V^\lambda(t_\eta,x_\eta)=\phi_\eta(t_\eta,x_\eta),
\]
and \(V^\lambda-\phi_\eta\) has a local minimum \(0\) at \((t_\eta,x_\eta)\), the viscosity supersolution property yields
\[
\mathcal{L}\phi_\eta(t_\eta,x_\eta)+
f_\lambda\bigl(t_\eta,x_\eta,\phi_\eta(t_\eta,x_\eta),\sigma(t_\eta,x_\eta)\nabla\phi_\eta(t_\eta,x_\eta)\bigr)\ge 0.
\]

Furthermore, since
\[
\nabla\rho_\eta(t_\eta,x_\eta)=0,
\quad\text{and}\quad
\mathcal{L}\rho_\eta(t_\eta,x_\eta)\ge 0,
\]
similar calculations as in Case 1 yield
\begin{align*}
&(\lambda-K_y)\ell_\eta(t_\eta,x_\eta)
\\\le\;&
-\Bigl(
\mathcal{L}p_b(t_\eta,x_\eta)
+
f\bigl(t_\eta,x_\eta,p_b(t_\eta,x_\eta),\sigma(t_\eta,x_\eta)\nabla p_b(t_\eta,x_\eta)\bigr)
\Bigr)\\
&\quad
+\varepsilon e^{-at_\eta}
\Bigl(
\mathcal{L}Q(x_\eta)-aQ(x_\eta)+K_y|Q(x_\eta)|+K_z|\sigma(t_\eta,x_\eta)\nabla Q(x_\eta)|
\Bigr).
\end{align*}
Again, by \eqref{eq:Qineq}, we obtain
\[
(\lambda-K_y)\ell_\eta(t_\eta,x_\eta)\le D_b^-.
\]

Note that since \(\rho_\eta(t_\eta,x_\eta)=0\), we have $\ell_\eta(t_\eta,x_\eta)=\ell(t_\eta,x_\eta)$,
and thus
\begin{equation*}
\begin{aligned}
(\lambda-K_y)\sup_K\ell\le&(\lambda-K_y)\ell(\bar{t},\bar{x})+(\lambda-K_y)\eta\\
\le&
(\lambda-K_y)\ell(t_\eta,x_\eta)+(\lambda-K_y)\eta\\
=&
(\lambda-K_y)\ell_\eta(t_\eta,x_\eta)+(\lambda-K_y)\eta\\
\le&
D_b^-+(\lambda-K_y)\eta.
\end{aligned}
\end{equation*}
Here, letting \(\eta\downarrow 0\) yields
\[
(\lambda-K_y)\sup_K\ell\le D_b^-.
\]
Since we assumed $\ell$ attains its global maximum on $K$, we conclude
\[
(\lambda-K_y)\Bigl((V^\lambda(t,x)-p_b(t,x))^--\varepsilon e^{-at}Q(x)\Bigr)\le D_b^-,
\qquad \forall (t,x)\in[0,T]\times\mathbb{R}^d.
\]
Again, letting \(\varepsilon\downarrow 0\) yields
\[
(\lambda-K_y)(V^\lambda(t,x)-p_b(t,x))^-\le D_b^-.
\]

\vspace{11pt}
\noindent This proves the first inequality.

For the second inequality, one argues in the same way with
\[
\hat\ell(t,x):=V^\lambda(t,x)-p_w(t,x)-\varepsilon e^{-at}Q(x),
\]
replacing positive maxima by negative minima throughout. This completes the proof.
 \end{proof}

\begin{corollary}
\label{stoapb}
    Suppose Assumptions \ref{assumptions},\ref{assumptionpde} and \ref{assumption2} hold. Let $(X,Y,Z)$ be the solution to \eqref{pBSDE} and choose $\lambda>2K_y$. Then, we have 
    $$\lambda (Y_t-p_b(t,X_t))^-\le 2D_b^-$$
    and
    $$\lambda (p_w(t,X_t)-Y_t)^-\le 2D_w^+$$
    almost surely for all $t\in[0,T]$.
\end{corollary}
\begin{proof}
    The result follows directly from Theorem~\ref{thmpde} and Proposition \ref{apb1}.
 \end{proof}

\begin{theorem}
\label{ydiffbound}
    Let $V$ and $V^\lambda$ each denote the solutions to \eqref{fbpde} and \eqref{pide}, respectively. If we choose $\lambda>2K_y$, we have
    $$-e^{K_y T}\frac{2D^+_w}{\lambda}\le V(t,x)-V^\lambda(t,x)\le e^{K_y T}\frac{2D_b^-}{\lambda}$$
    for all $(t,x)\in[0,T]\times\mathbb{R}^d$. Thus, if we denote by $(X,Y,Z,A,K)$ the solution to \eqref{BSDE} and by $(X,Y^\lambda,Z^\lambda)$ the solution to \eqref{pBSDE}, then we have
    \begin{equation}
    \label{boundineq}|Y_t-Y_t^\lambda|\le e^{K_y T}\frac{2\max(D^+_w,D^-_b)}{\lambda}
    \end{equation}
    almost surely for all $t\in[0,T]$.
\end{theorem}
\begin{proof}
For brevity, we write
\[
C_b:=\frac{D_b^-}{\lambda-K_y},
\qquad
C_w:=\frac{D_w^+}{\lambda-K_y}.
\]
Note that for $\lambda>2K_y$, we have $C_b\le e^{K_y T}\frac{2D_b^-}{\lambda}$ and $-C_w\ge -e^{K_y T}\frac{2D^+_w}{\lambda}$.

We first prove the upper bound inequality
\[
V(t,x)-V^\lambda(t,x)\le e^{K_y(T-t)}C_b.
\]
To this end, let
\[
c_b(t):=e^{K_y(T-t)}C_b,
\qquad
\overline V^\lambda(t,x):=V^\lambda(t,x)+c_b(t),
\]
then, it is sufficient to show that \(\overline V^\lambda\) is a viscosity supersolution of
\begin{equation}\label{eq:obsPDE}
\max\Bigl(
\min\bigl(-\mathcal{L}u-f(t,x,u,\sigma(t,x)\nabla_x u),\,u-p_b(t,x)\bigr),\,
u-p_w(t,x)
\Bigr)=0.
\end{equation}

Let \(\varphi\in C^{1,2}\) be such that \(\overline V^\lambda-\varphi\) attains a local minimum \(0\) at \((t_0,x_0)\), and set
\[
\psi(t,x):=\varphi(t,x)-c_b(t).
\]
Then it is clear that \(V^\lambda-\psi\) attains a local minimum \(0\) at \((t_0,x_0)\).

\noindent We consider the following two cases.

\medskip
\noindent
\textbf{Case 1.}
If \(\overline V^\lambda(t_0,x_0)\ge p_w(t_0,x_0)\), then the supersolution condition for \eqref{eq:obsPDE} is immediate since the second entry in the outer maximum is nonnegative.

\medskip
\noindent
\textbf{Case 2.}
Assume that \(\overline V^\lambda(t_0,x_0)<p_w(t_0,x_0)\). The estimate of Proposition~\ref{apb1},
\[
(\lambda-K_y)(V^\lambda-p_b)^-\le D_b^-,
\]
implies that
\[
(V^\lambda-p_b)^-\le C_b,
\]
or equivalently,
\[
V^\lambda(t,x)+C_b\ge p_b(t,x)
\qquad \forall (t,x)\in[0,T]\times\mathbb{R}^d.
\]
Since \(c_b(t)\ge C_b\), one has
\[
\overline V^\lambda(t_0,x_0)\ge p_b(t_0,x_0).
\]
In this case, verifying the supersolution condition reduces to proving
\[
-\mathcal{L}\varphi(t_0,x_0)-f\bigl(t_0,x_0,\varphi(t_0,x_0),\sigma(t_0,x_0)\nabla_x\varphi(t_0,x_0)\bigr)\ge 0.
\]

Since \(V^\lambda\) is a viscosity supersolution of the penalized PDE, we have
\begin{equation}
\label{T5:1}
-\mathcal{L}\psi(t_0,x_0)-f_\lambda\bigl(t_0,x_0,\psi(t_0,x_0),\sigma(t_0,x_0)\nabla_x\psi(t_0,x_0)\bigr)\ge 0.
\end{equation}
Also, since \(\varphi=\psi+c_b\) and \(c_b\) depends only on \(t\), it follows that
\begin{equation}
\label{T5:2}
-\mathcal{L}\varphi(t_0,x_0)= -\mathcal{L}\psi(t_0,x_0)-c_b'(t_0).
\end{equation}
Furthermore, the inequality \(\overline V^\lambda(t_0,x_0)<p_w(t_0,x_0)\) gives us
\[
\psi(t_0,x_0)=V^\lambda(t_0,x_0)<p_w(t_0,x_0)-c_b(t_0)<p_w(t_0,x_0),
\]
and hence
\[
(p_w(t_0,x_0)-\psi(t_0,x_0))^-=0.
\]
Thus, we arrive at
\begin{equation}
\label{T5:3}
f_\lambda\bigl(t_0,x_0,\psi,\sigma\nabla\psi\bigr)
=
f\bigl(t_0,x_0,\psi,\sigma\nabla\psi\bigr)
+\lambda(\psi-p_b)^-(t_0,x_0)
\ge
f\bigl(t_0,x_0,\psi,\sigma\nabla\psi\bigr).
\end{equation}
Combining \eqref{T5:1}, \eqref{T5:2}, and \eqref{T5:3}, we obtain
\begin{align*}
&-\mathcal{L}\varphi(t_0,x_0)-f\bigl(t_0,x_0,\varphi(t_0,x_0),\sigma\nabla\varphi(t_0,x_0)\bigr)\\
&\ge
-c_b'(t_0)
+
f\bigl(t_0,x_0,\psi(t_0,x_0),\sigma\nabla\psi(t_0,x_0)\bigr)
-
f\bigl(t_0,x_0,\varphi(t_0,x_0),\sigma\nabla\varphi(t_0,x_0)\bigr).
\end{align*}
Because \(\nabla\varphi=\nabla\psi\) and \(\varphi=\psi+c_b)\), the Lipschitz continuity of \(f\) in the \(y\)-variable gives
\[
f\bigl(t_0,x_0,\psi,\sigma\nabla\psi\bigr)
-
f\bigl(t_0,x_0,\varphi,\sigma\nabla\varphi\bigr)
\ge -K_yc_b(t_0).
\]
Hence
\[
-\mathcal{L}\varphi(t_0,x_0)-f\bigl(t_0,x_0,\varphi(t_0,x_0),\sigma\nabla\varphi(t_0,x_0)\bigr)
\ge -c_b'(t_0)-K_yc_b(t_0).
\]
Since \(c_b'(t)=-K_yc_b(t)\), the right-hand side is equal to \(0\). This proves that \(\overline V^\lambda\) is a viscosity supersolution of \eqref{eq:obsPDE}.

The terminal condition satisfies
\[
\overline V^\lambda(T,x)=V^\lambda(T,x)+C_b=g(x)+C_b\ge g(x)=V(T,x),
\]
and hence applying the comparison principle gives us
\[
V(t,x)\le \overline V^\lambda(t,x)=V^\lambda(t,x)+e^{K_y(T-t)}C_b,
\]
which implies the desired conclusion.

The lower bound can also be proven similarly. 

\end{proof}

For the total penalty error, we define increasing finite variation processes $A^\lambda$ and $K^\lambda$ by
$$A^\lambda_t:=\lambda\int^{t}_0(Y^\lambda_s-p_b(s,X_s))^-\,ds,\quad K^\lambda_t:=\lambda\int^t_0(p_w(s,X_s)-Y^\lambda_s)^-\,ds.$$
Then the backward equation of penalized BSDE \eqref{pBSDE} can be rewritten as
$$-dY^\lambda_t=f(t,X_t,Y^\lambda_t,Z^\lambda_t)\,dt+dA^\lambda_t-dK^\lambda_t-(Z^\lambda_t)^\top\,dW_t$$
in differential form.
\begin{theorem}
\label{penaltyerror}
    Let us denote by $(X,Y,Z,A,K)$ and $(X,Y^\lambda,Z^\lambda)$ the solutions to \eqref{BSDE} and \eqref{pBSDE}, respectively. Suppose Assumptions~\ref{assumptions},\ref{assumptionpde}, and \ref{assumption2} hold and choose $\lambda>2K_y$. Then, we have
    $$\mathbb{E}\left[\sup_{0\le t\le T}|Y_t-Y^\lambda_t|^2\right]<\frac{C_2}{\lambda^2}\quad\text{ and }\quad\int^T_0\mathbb{E}|Z_s-Z^\lambda_s|^2\,ds<\frac{C_2}{\lambda}.$$
    
\end{theorem}
\begin{proof}
First, we have
$$\mathbb{E}\left[\sup_{0\le t\le T}|Y_t-Y^\lambda_t|^2\right]<\frac{C_2}{\lambda^2}$$
from Theorem~\ref{ydiffbound}. Now, let us denote $\delta Y^\lambda:=Y^\lambda-Y$ and $\delta Z^\lambda:=Z^\lambda-Z$. Then, we have
\begin{equation*}
\begin{aligned}-\delta Y^\lambda_t=& \left(f(t,X_t,Y^\lambda_t,Z^\lambda_t)-f(t,X_t,Y_t,Z_t)\right)\,dt\\&+(dA^\lambda_t-dA_t)-(dK^\lambda_t-dK_t)-(\delta Z^\lambda_t)^\top\,dW_t.
\end{aligned}
\end{equation*}
Thus, applying the It\^o formula to $|\delta Y^\lambda_t|^2$, we obtain
\begin{equation}
\label{energy}
\begin{aligned}
|\delta Y^\lambda_0|^2
&\;+\; \int^T_0 |\delta Z^\lambda_s|^2 \, ds \\
&= |\delta Y^\lambda_T|^2
  + 2 \int^T_0 \delta Y^\lambda_s \, \left(f(s,X_s,Y^\lambda_s,Z^\lambda_s)-f(s,X_s,Y_s,Z_s)\right) \, ds\\
  &\quad+ 2 \int^T_0 \delta Y^\lambda_s \, dA^\lambda_s
  - 2 \int^T_0 \delta Y^\lambda_s \, dA_s \\
&\quad
  - 2 \int^T_0 \delta Y^\lambda_s \, dK^\lambda_s
  + 2 \int^T_0 \delta Y^\lambda_s \, dK_s
  - 2 \int^T_0 \delta Y^\lambda_s \, \Delta Z_s \, dW_s.
\end{aligned}
\end{equation}
By Young's inequality and the Lipschitz property of $f$, we have
\begin{equation}
\label{1st}
    \begin{aligned}
        &2 \int^T_0 \delta Y^\lambda_s \, \left(f(s,X_s,Y^\lambda_s,Z^\lambda_s)-f(s,X_s,Y_s,Z_s)\right) \, ds\\
        \le&4K^2_z\int^T_0|\delta Y^\lambda_s|^2\,ds+\frac{1}{4K^2_z}\int^T_0\left(f(s,X_s,Y^\lambda_s,Z^\lambda_s)-f(s,X_s,Y_s,Z_s)\right)^2\,ds\\
        \le&4K^2_z\int^T_0|\delta Y^\lambda_s|^2\,ds +\frac{1}{4K^2_z}\int^T_0 2K^2_y|\delta Y^\lambda_s|^2+2K^2_Z|\delta Z^\lambda_s|^2\,ds\\
        \le& \frac{C_2}{\lambda^2}+\frac{1}{2}\int^T_0 |\delta Z^\lambda_s|^2 \, ds.
    \end{aligned}
\end{equation}
Next, note that
\begin{equation}
    \label{2nd}
    \begin{aligned}
    2\int^T_0\delta Y^\lambda_s dA^\lambda_s=&2\int^T_0(Y^\lambda_s-p_b(s,X_s)) dA^\lambda_s-2\int^T_0(Y_s-p_b(s,X_s)) dA^\lambda_s\\
    \le&2\int^T_0(Y^\lambda_s-p_b(s,X_s))(Y^\lambda_s-p_b(s,X_s)) ^-ds\\
    \le&0
    \end{aligned}
\end{equation}
and
\begin{equation}
    \label{3rd}
    \begin{aligned}
    2\int^T_0\delta Y^\lambda_s dK^\lambda_s=&2\int^T_0(Y^\lambda_s-p_w(s,X_s)) dK^\lambda_s-2\int^T_0(Y_s-p_w(s,X_s)) dK^\lambda_s\\
    \ge&-2\int^T_0(p_w(s,X_s)-Y^\lambda_s)(p_w(s,X_s)-Y^\lambda_s) ^-ds\\
    \ge&0.
    \end{aligned}
\end{equation}
Plugging in \eqref{1st}, \eqref{2nd}, and \eqref{3rd} into \eqref{energy}, and then taking expectation of absolute values of both sides, we arrive at

\begin{equation}
\label{energy2}
\begin{aligned}
\mathbb{E}|\delta Y^\lambda_0|^2
&\;+\;\frac{1}{2} \int^T_0 \mathbb{E}|\delta Z^\lambda_s|^2 \, ds \\
&\le \frac{C_2}{\lambda^2}
  + 2 \mathbb{E}\left[\int^T_0 |\delta Y^\lambda_s| \, dA_s\right]
  + 2 \mathbb{E}\left[\int^T_0 |\delta Y^\lambda_s| \, dK_s\right].
\end{aligned}
\end{equation}
Also, since \(A,K\in A^2\) we have \(\mathbb{E}[A_T^2]+\mathbb{E}[K_T^2]<\infty\). Thus, applying the Cauchy--Schwarz inequality yields
\begin{align*}
\mathbb{E}\int_0^T |\delta Y_s^\lambda|\,dA_s
&\le
\left(\mathbb{E}\left[\sup_{0\le s\le T}|\delta Y_s^\lambda|^2\right]\right)^{1/2}
\left(\mathbb{E}[A_T^2]\right)^{1/2}
\le \frac{C_2}{\lambda}, \\
\mathbb{E}\int_0^T |\delta Y_s^\lambda|\,dK_s
&\le
\left(\mathbb{E}\left[\sup_{0\le s\le T}|\delta Y_s^\lambda|^2\right]\right)^{1/2}
\left(\mathbb{E}[K_T^2]\right)^{1/2}
\le \frac{C_2}{\lambda}.
\end{align*}
Plugging these estimates into \eqref{energy2}, we conclude that
\[
\int_0^T \mathbb{E}|\delta Z_s^\lambda|^2\,ds \le \frac{C_2}{\lambda}.
\]
 \end{proof}

\section{Numerical scheme and error analysis}
\label{sec:ea}
We describe the numerical scheme and derive the error bounds of our method. From here on, we denote by $C_2$ any constant depending only on the data of the BSDE \eqref{BSDE}: $T$, $X_0$, $b$, $\sigma$, $g$, $f$, $p_b$, and $p_w$. Note that the constant $C_2$ does not depend on the penalty parameter $\lambda$. 

\subsection{Numerical scheme}
\label{subs:ns}
We describe the numerical scheme of our approximation method. First, we fix two time grids $\pi:=\{t_i:=iT/n, i=0,...,n\}$ and $\tilde{\pi}:=\{t_i:=iT/(mn), i=0,...,mn\}$ $(n,m\in\mathbb{N})$ on $[0,T]$ and write $\Delta t:= T/n$ and $\Delta\tilde{t}:=T/(mn)$. Note that $\pi\subset\tilde{\pi}$, $\Delta t\ge\Delta \tilde{t}$, and the generic constant $C_2$ does not depend on the resolution $n$ nor on $m$.

The solution $X$ to the forward equation is approximated by the Euler--Maruyama scheme $X^{\pi}$, with values on the finer grid  $\tilde{\pi}$ given inductively as
\begin{equation}
\label{emx}
    \begin{cases}
        X^{\tilde{\pi}}_0 = X_0\\
        X^{\tilde{\pi}}_{\tilde{t}_{i+1}}=X^{\tilde{\pi}}_{\tilde{t}_i}+b(\tilde{t}_i,X^{\tilde{\pi}}_{\tilde{t}_i})\Delta \tilde{t}+\sigma(\tilde{t}_i,X^{\tilde{\pi}}_{\tilde{t}_i})\left(W_{\tilde{t}_{i+1}}-W_{\tilde{t}_i}\right).
    \end{cases}
\end{equation}
for $i=0,...,mn-1$.

Then, on the coarser grid $\pi$, we consider the discretization scheme
\begin{equation}\label{scheme}
\mathcal{Y}^{\lambda,\pi}_{t_{i+1}}
:=
\mathcal{Y}^{\lambda,\pi}_{t_i}
-
f_\lambda\left(t_i,X^{\tilde{\pi}}_{t_i},\mathcal{Y}^{\lambda,\pi}_{t_i},\mathcal{Z}^{\lambda,\pi}_{t_i}\right)\Delta t
+
\mathcal{Z}^{\lambda,\pi}_{t_i}\Delta W_{i+1},
\end{equation}
with terminal condition
\begin{equation}\label{discterm}
\mathcal{Y}^{\lambda,\pi}_{t_n}=g(X^{\tilde{\pi}}_{t_n}),
\end{equation}
where $\Delta W_{i+1}:=W_{t_{i+1}}-W_{t_i}$, to approximate the penalized BSDE solution $(Y^\lambda,Z^\lambda)$ associated with the forward Euler approximation $X^{\tilde{\pi}}$ of $X$.

It is easy to see that \eqref{scheme} implies the implicit discretization scheme
\begin{equation}\label{discretization}
\begin{cases}
\mathcal{Y}^{\lambda,\pi}_{t_i}
:=
\mathbb{E}\left[\mathcal{Y}^{\lambda,\pi}_{t_{i+1}}\mid\mathcal{F}_{t_i}\right]
+
f_\lambda\left(t_i,X^{\tilde{\pi}}_{t_i},\mathcal{Y}^{\lambda,\pi}_{t_i},\mathcal{Z}^{\lambda,\pi}_{t_i}\right)\Delta t,\\[0.5em]
\mathcal{Z}^{\lambda,\pi}_{t_i}
:=
\frac{1}{\Delta t}\mathbb{E}\left[\mathcal{Y}^{\lambda,\pi}_{t_{i+1}}\Delta W_{i+1}\mid\mathcal{F}_{t_i}\right],
\end{cases}
\end{equation}
for $i=0,\ldots,n-1$, as introduced for instance in \cite{BT:04} and \cite{BE:08}.

It is well known that under our assumptions,
\begin{equation}
\label{xdifbound}
\max_{i\in\{0,\cdots,mn\}}\mathbb{E}\left[\sup_{t\in[\tilde{t}_i,\tilde{t}_{i+1})}\left|X_t-X^{\tilde{\pi}}_{\tilde{t}_i}\right|^p\right]\le C_2\, \Delta \tilde{t}^\frac{p}{2},
\end{equation}
\begin{equation}
\label{xebound}
\max_{i\in\{0,\cdots,n\}}\mathbb{E}\left|X^{\tilde{\pi}}_{t_i}\right|^p\le C_2,
\end{equation}
\begin{equation}
    \label{xbound}
        \|X\|^p_{\mathbb{S}^p}<C_2,
    \end{equation}
and
\begin{equation}
\label{xdifbound2}
\mathbb{E}\left|X_s-X_{t_{i}}\right|^p\le C_2\, \Delta t^\frac{p}{2}\quad(s\in[t_i,t_{i+1}), i=0,\cdots,n-1)
\end{equation}
hold for any $p\ge 2$ \citep{O:03,KP:92}.

\subsection{A priori estimates}

Before obtaining the error bounds, we derive a priori estimates. The following bound on the norms of the solutions to \eqref{pBSDE} can be obtained by applying previous results and adapting standard proofs of well-known results such as Lemmata 1.5.1 and 1.5.2 of Part II of \cite{E:06}.

\begin{proposition}
\label{xyzbound}
    Suppose Assumptions \ref{assumptions}, \ref{assumptionpde}, and \ref{assumption2} hold. Let $(X,Y^\lambda,Z^\lambda)$ denote the solution to \eqref{pBSDE} and choose $\lambda>2K_y$. Then, we have    $$\|Y^\lambda\|^p_{\mathbb{S}^p}+\|Z^\lambda\|^p_{\mathbb{H}^p}\le C_2$$
    for any $p\ge2$.
\end{proposition}
\begin{proof}
     We first divide $[0,T]$ into $N$ intervals $[\tau_i,\tau_{i+1}]$ with $\tau_{i+1}-\tau_i=1/N\,(i=0,\cdots,N-1)$. For any $\tau_i\le r\le \tau_{i+1}$, we have
    $$Y^\lambda_r=Y^\lambda_{\tau_{i+1}}+\int^{\tau_{i+1}}_r f_\lambda(s,X_s,Y^\lambda_s,Z^\lambda)\,ds-\int^{\tau_{i+1}}_r(Z^\lambda_s)^\top\,dW_s.$$
    Taking conditional expectations on both sides leads to    $$Y^\lambda_r=\mathbb{E}\left[Y^\lambda_{\tau_{i+1}}+\int^{\tau_{i+1}}_r f_\lambda(s,X_s,Y^\lambda_s,Z^\lambda_s)\,ds\mid \mathcal{F}_{r}\right].$$
    Thus, for $\tau_i\le t \le r\le \tau_{i+1},$ we obtain
    $$\left|Y^\lambda_r\right|\le\mathbb{E}\left[\left|Y^\lambda_{\tau_{i+1}}\right|+\int^{\tau_{i+1}}_t \left|f_\lambda(s,X_s,Y^\lambda_s,Z^\lambda_s)\right|\,ds\mid\mathcal{F}_{r}\right].$$
    Thus, by Doob and Jensen's inequality, we obtain
    \begin{equation}\label{ypre}\begin{aligned}&\|Y^\lambda\|^p_{\mathbb{S^p}[t,\tau_{i+1}]}:=\mathbb{E}\left[\sup_{t\le r\le \tau_{i+1}}\left|Y^\lambda_r\right|^p\right]\\\le& C_2\,\mathbb{E}\left[\left|Y^\lambda_{\tau_{i+1}}\right|^p+\left(\int^{\tau_{i+1}}_t \left|f_\lambda(s,X_s,Y^\lambda_s,Z^\lambda_s)\right|\,ds\right)^p\right].\end{aligned}\end{equation} Next, we apply the Burkholder--Davis--Gundy inequality to the martingale
    $$M_u:=\int^u_t(Z^\lambda_s)^\top\,dW_s=Y^\lambda_u-Y^\lambda_t+\int^u_tf_\lambda\left(s,X_s,Y^\lambda_s,Z^\lambda_s\right)\,ds\quad(t\le u\le\tau_{i+1})$$
    to obtain
        \begin{equation}
        \label{zpre}
        \begin{aligned}
        &\|Z\|^p_{\mathbb{H}^p[t,\tau_{i+1}]}
        :=\mathbb{E}\left[\left(\int^{\tau_{i+1}}_t(Z^\lambda_s)^2\,ds\right)^\frac{p}{2}\right]
        \\=&\mathbb{E}\left[[M]^\frac{p}{2}_{\tau_{i+1}}\right]
        \le C_2\,\mathbb{E}\left[\left(\sup_{t\le r\le \tau_{i+1}}|M_r|\right)^p\right]\\
        \le& C_2\,\mathbb{E}\left[\sup_{t\le r\le \tau_{i+1}} \left|Y^\lambda_r\right|^p+\left|Y^\lambda_t\right|^p+\int^{\tau_{i+1}}_t \left|f_\lambda\left(s,X_s,Y^\lambda_s,Z^\lambda_s\right)\right|^p\,ds\right]\\
        \le & C_2\left(\|Y^\lambda\|^p_{\mathbb{S}^p[t,\tau_{i+1}]}+\,\mathbb{E}\left[\int^{\tau_{i+1}}_t \left|f_\lambda\left(s,X_s,Y^\lambda_s,Z^\lambda_s\right)\right|^p\,ds\right]\right).
        \end{aligned}
        \end{equation}

    Thus, putting together \eqref{ypre} and \eqref{zpre} yields
    \begin{equation*}\begin{aligned}&\|Y^\lambda\|^p_{\mathbb{S}^p[t,\tau_{i+1}]}+\|Z^\lambda\|^p_{\mathbb{H}^p[t,\tau_{i+1}]}\\\le& C_2\,\mathbb{E}\left[\left|Y^\lambda_{\tau_{i+1}}\right|^p+\left(\int^{\tau_{i+1}}_t \left|f_\lambda(s,X_s,Y^\lambda_s,Z^\lambda_s)\right|\,ds\right)^p\right].\end{aligned}\end{equation*}
    We then apply the continuity properties of $f$, inequality \eqref{xbound}, and Corollary~\ref{stoapb} to obtain
    \begin{equation*}
        \begin{aligned}
            &\|Y^\lambda\|^p_{\mathbb{S}^p[t,\tau_{i+1}]}+\|Z^\lambda\|^p_{\mathbb{H}^p[t,\tau_{i+1}]}\\\le& C_2\,\mathbb{E}\left[|Y^\lambda_{\tau_{i+1}}|^p+\left(\int^{\tau_{i+1}}_t \left|f_\lambda(s,X_s,Y^\lambda_s,Z^\lambda_s)\right|\,ds\right)^p\right]\\
            \le& C_2\mathbb{E}\Bigg[|Y^\lambda_{\tau_{i+1}}|^p+\Big(\int^{\tau_{i+1}}_t \left|f(s,X_s,Y^\lambda_s,Z^\lambda_s)\right|+\\
            &\qquad\left|\lambda\left((Y^\lambda_s-p_b(s,X_s))^--(p_w(s,X_s)-Y^\lambda_s)^-\right)\right|\,ds\Big)^p\Bigg]\\
            \le& C_2\mathbb{E}\left[|Y^\lambda_{\tau_{i+1}}|^p+\left(\int^{\tau_{i+1}}_t \left|f(s,X_s,Y^\lambda_s,Z^\lambda_s)\right|\,ds\right)^p\right]\\
            \le& C_2\mathbb{E}\left(|Y^\lambda_{\tau_{i+1}}|^p+\int^{\tau_{i+1}}_t |f(0,0,0,0)|^p+\mathbb{E}|Y^\lambda_s|^p+\mathbb{E}|Z^\lambda_s|^p\,ds\right)\\
            &+C_2\mathbb{E}\left[\int^{\tau_{i+1}}_{t} s^\frac{p}{2}+|X_s|^p\,ds\right]\\
            \le&C_2\left(\mathbb{E}|Y^\lambda_{\tau_{i+1}}|^p+\int^{\tau_{i+1}}_t \|Y^\lambda\|^p_{\mathbb{S}^p[t,\tau_{i+1}]}\,ds+\frac{1}{N}\|Z^\lambda\|^p_{\mathbb{H}^p{[t,\tau_{i+1}]}}\right).
        \end{aligned}
    \end{equation*}
    Choosing $N>2C_2$, we arrive at
    \begin{equation}\label{yfinal}\begin{aligned}\|Y^\lambda\|^p_{\mathbb{S}^p[t,\tau_{i+1}]}\le&\|Y^\lambda\|^p_{\mathbb{S}^p[t,\tau_{i+1}]}+\|Z^\lambda\|^p_{\mathbb{H}^p[t,\tau_{i+1}]}\\\le& C_2\left(\mathbb{E}|Y^\lambda_{\tau_{i+1}}|^p+\int^{\tau_{i+1}}_t \|Y^\lambda\|^p_{\mathbb{S}^p[t,\tau_{i+1}]}\,ds\right),\end{aligned}\end{equation}
    Therefore, applying Gr\"onwall's lemma yields
    \begin{equation*}\|Y^\lambda\|^p_{\mathcal{S}^p[\tau_i,\tau_{i+1}]}\le C_2\mathbb{E}|Y^\lambda_{\tau_{i+1}}|^p\le C_2\|Y^\lambda\|^p_{\mathcal{S}^p[\tau_{i+1},\tau_{i+2}]}\quad(i=0,\cdots,N-2).
    \end{equation*}
    Thus, $\|Y^\lambda\|^p_{\mathbb{S}^p}<C_2$ follows from induction, the Lipschitz property of $\bar{g}$, and \eqref{xbound}. Plugging this into \eqref{yfinal} and using induction yields
    $\|Z^\lambda\|^p_{\mathbb{H}^p}<C_2$ as well, and the desired conclusion follows.
 \end{proof}

\begin{remark}
    The intervals $[\tau_i,\tau_{i+1}]$ introduced in the proof of Proposition~\ref{xyzbound} are irrelevant to the partition $\pi$ of the discretization scheme of our method.
\end{remark}

\begin{corollary}
\label{apbcor}
    Suppose Assumptions \ref{assumptions}, \ref{assumptionpde}, and \ref{assumption2} hold. Let $(X,Y^\lambda,Z^\lambda)$ denote the solution to \eqref{pBSDE} and choose $\lambda>2K_y$. Then,
        $$\sum^{n-1}_{i=0}\int^{t_{i+1}}_{t_i}\mathbb{E}\left|Y^\lambda_s-Y^\lambda_{t_i}\right|^2\,ds\le C_2\Delta t.$$
\end{corollary}
\begin{proof}
    From \eqref{pBSDE}, we calculate
    \begin{equation*}
        \begin{aligned}
            &\mathbb{E}\left|Y^\lambda_s-Y^\lambda_{t_i}\right|^2\\
            \le& C_2\,\mathbb{E}\left[\int^s_{t_i}\left|f_\lambda\left(r,X_r,Y^\lambda_r,Z^\lambda_r\right)\right|^2\,dr+\left(\int^s_{t_i}(Z^\lambda_r)^\top\,dW_r\right)^2\right]\\
            \le& C_2\,\Bigg(\int^s_{t_i} \mathbb{E}\left[\left|f\left(r,X_r,Y^\lambda_r,Z^\lambda_r\right)\right|^2+(Z^\lambda_r)^2\right]\,dr\\
            &\quad+\int^s_{t_i}\mathbb{E}\left|\lambda\left((Y^\lambda_r-p_b(r,X_r))^--(p_w(r,X_r)-Y^\lambda_r)^-\right)\right|^2\,dr\Bigg)\\
            \le& C_2\,\int^s_{t_i} \left|f(0,0,0,0)\right|^2+r+\mathbb{E}|X_r|^2+\mathbb{E}|Y^\lambda_r|^2+\mathbb{E}|Z^\lambda_r|^2\,dr\\
            =& C_2\left((s-t_i)+\int^s_{t_i}\mathbb{E}|Z^\lambda_r|^2\,dr\right),
        \end{aligned}
    \end{equation*}
for $t_i\le s\le t_{i+1}$, where we have used It\^o isometry in the second inequality and Corollary~\ref{stoapb}, the continuity property of $f$, and Proposition~\ref{xyzbound} in the last inequality. Thus, using Fubini's theorem and Proposition~\ref{xyzbound} again, we obtain
\begin{equation*}
    \begin{aligned}
        &\sum^{n-1}_{i=0}\int^{t_{i+1}}_{t_i}\mathbb{E}\left|Y^\lambda_s-Y^\lambda_{t_i}\right|^2\,ds\\
        &\le C_2\sum^{n-1}_{i=0}\int^{t_{i+1}}_{t_i}\left((s-t_i)+\int^s_{t_i}\mathbb{E}|Z^\lambda_r|^2\,dr\right)\,ds\\
        &\le C_2\sum^{n-1}_{i=0}\frac{1}{2}(t_{i+1}-t_i)^2+\sum^{n-1}_{i=0}\int^{t_{i+1}}_{t_i}\int^s_{t_i}\mathbb{E}|Z^\lambda_r|^2\,dr\,ds\\
        &\le C_2\Delta t+ \sum^{n-1}_{i=0}\int^{t_{i+1}}_{t_i}(t_{i+1}-s)\mathbb{E}|Z^\lambda_s|^2\,ds\\
        &\le C_2\Delta t+\Delta t\sum^{n-1}_{i=0}\int^{t_{i+1}}_{t_i}\mathbb{E}|Z^\lambda_s|^2\,ds\\
        &=C_2\Delta t+\Delta t\|Z^\lambda\|^2_{\mathbb{H}^2}\\
        &\le C_2\Delta t.
    \end{aligned}
\end{equation*}
 \end{proof}

\subsection{Square error bound}
We now derive the total square error bound of the proposed scheme.
\begin{theorem} 
\label{discerror}
     Suppose Assumptions \ref{assumptions},\ref{assumptionpde} and \ref{assumption2} hold. Let $(X,Y^\lambda,Z^\lambda)$ denote the solution to \eqref{pBSDE} and choose $\lambda>2K_y$ and $\Delta t<1/(16K^2_z)$. Then, for $(X^{\tilde{\pi}},\mathcal{Y}^{\lambda,\pi},\mathcal{Z}^{\lambda,\pi})$ satisfying \eqref{emx} and \eqref{scheme}, we have
     \begin{equation*}
         \begin{aligned}
  \max_{i\in\{0,\cdots,n-1\}}\mathbb{E}\left|Y^\lambda_{t_i}-\mathcal{Y}^{\lambda,\pi}_{t_i}\right|^2+\sum^{n-1}_{i=0}\mathbb{E}\left[\int^{t_{i+1}}_{t_i}|Z^\lambda_s-\mathcal{Z}^{\lambda,\pi}_{t_i}|^2\,ds\right]\\
        \le C_2 \left(\mathbb{E}\left|g(X_T)-\mathcal{Y}^{\lambda,\pi}_T\right|^2+\lambda^2\Delta t\right).
        \end{aligned}
        \end{equation*}
    Furthermore, if $(X^{\tilde{\pi}},\mathcal{Y}^{\lambda,\pi},\mathcal{Z}^{\lambda,\pi})$ also satisfies \eqref{discterm}, then
         $$\max_{i\in\{0,\cdots,n-1\}}\mathbb{E}\left|Y^\lambda_{t_i}-\mathcal{Y}^{\lambda,\pi}_{t_i}\right|^2+\sum^{n-1}_{i=0}\mathbb{E}\left[\int^{t_{i+1}}_{t_i}|Z^\lambda_s-\mathcal{Z}^{\lambda,\pi}_{t_i}|^2\,ds\right]\\
        \le C_2 \lambda^2\Delta t.$$
\end{theorem}
\begin{proof}
    We use the notations $\delta Y^\lambda_{t_i}:=Y^\lambda_{t_i}-\mathcal{Y}^{\lambda,\pi}_{t_i}$, $\delta X_{t_i}:=X_{t_i}-X^{\tilde{\pi}}_{t_i}$ and
    $$\delta f_{t_i}(s):=f\left(t_i,X_{t_i},Y^\lambda_{t_i},Z^\lambda_s\right)-f\left(t_i,X^{\tilde{\pi}}_{t_i},\mathcal{Y}^{\lambda,\pi}_{t_i},\mathcal{Z}^{\lambda,\pi}_{t_i}\right).$$
    From \eqref{pBSDE}, we have
    \begin{equation}
    \label{contdecomp}
        \begin{aligned}
            Y^\lambda_{t_i}=&Y^\lambda_{t_{i+1}}+\int^{t_{i+1}}_{t_i}f_\lambda\left(s,X_s,Y^\lambda_s,Z^\lambda_s\right)\,ds-\int^{t_{i+1}}_{t_i}(Z^\lambda_s)^\top\,dW_s\\
            =&Y^\lambda_{t_{i+1}}+\int^{t_{i+1}}_{t_i}f_\lambda\left(t_i,X_{t_i},Y^\lambda_{t_i},Z^\lambda_s\right)\,ds\\
            &+\int^{t_{i+1}}_{t_i}f_\lambda\left(s,X_s,Y^\lambda_s,Z^\lambda_s\right)-f_\lambda\left(t_i,X_{t_i},Y^\lambda_{t_i},Z^\lambda_s\right)\,ds-\int^{t_{i+1}}_{t_i}(Z^\lambda_s)^\top\,dW_s\\
            =&Y^\lambda_{t_{i+1}}+\int^{t_{i+1}}_{t_i}f\left(t_i,X_{t_i},Y^\lambda_{t_i},Z^\lambda_s\right)\,ds\\
            &+\lambda\left((Y^\lambda_{t_i}-p_b(t_i,X_{t_i}))^--(p_w(t_i,X_{t_i})-Y^\lambda_{t_i})^-\right)\Delta t-\int^{t_{i+1}}_{t_i}(Z^\lambda_s)^\top\,dW_s+r_i,
        \end{aligned}
    \end{equation}
    where 
    $$r_i:=\int^{t_{i+1}}_{t_i}f_\lambda\left(s,X_s,Y^\lambda_s,Z^\lambda_s\right)-f_\lambda\left(t_i,X_{t_i},Y^\lambda_{t_i},Z^\lambda_s\right)\,ds.$$

Thus, combining \eqref{scheme} and \eqref{contdecomp} yields
\begin{equation*}
    \begin{aligned}
        &\delta Y^\lambda_{t_i}\\
        =&\delta Y^\lambda_{t_{i+1}}+\int^{t_{i+1}}_{t_i}\delta f_{t_i}(s)\,ds+\lambda\left((Y^\lambda_{t_i}-p_b(t_i,X_{t_i}))^--(\mathcal{Y}^{\lambda,\pi}_{t_i}-p_b(t_i,X^{\tilde{\pi}}_{t_i}))^-\right)\Delta t\\
        &+\lambda\left((p_w(t_i,X_{t_i})-Y^\lambda_{t_i})^--(p_w(t_i,X^{\tilde{\pi}}_{t_i})-\mathcal{Y}^{\lambda,\pi}_{t_i})^-\right)\Delta t \\
        &\qquad\quad-\int^{t_{i+1}}_{t_i}(Z^\lambda_s-\mathcal{Z}^{\lambda,\pi}_{t_i})^\top\,dW_s+r_i.
    \end{aligned}
\end{equation*}
Let
$$\eta_1:=-\lambda \Delta t\frac{\left(Y^\lambda_{t_i}-p_b(t_i,X_{t_i})\right)^--\left(\mathcal{Y}^{\lambda,\pi}_{t_i}-p_b(t_i,X_{t_i})\right)^-}{\delta Y^\lambda_{t_i}}\mathbf{1}_{\delta Y^\lambda_{t_i}\ne0}\ge 0,$$
$$\eta_2:=\lambda \Delta t\frac{\left(p_w(t_i,X_{t_i})-Y^\lambda_{t_i}\right)^--\left(p_w(t_i,X_{t_i})-\mathcal{Y}^{\lambda,\pi}_{t_i}\right)^-}{\delta Y^\lambda_{t_i}}\mathbf{1}_{\delta Y^\lambda_{t_i}\ne0}\ge 0,$$
and $\eta=\eta_3+\eta_4$, where
$$\eta_3:=(\mathcal{Y}^{\lambda,\pi}_{t_i}-p_b(t_i,X_{t_i}))^--(\mathcal{Y}^{\lambda,\pi}_{t_i}-p_b(t_i,X^{\tilde{\pi}}_{t_i}))^-$$
and
$$\eta_4:=(p_w(t_i,X_{t_i})-\mathcal{Y}^{\lambda,\pi}_{t_i})^--(p_w(t_i,X^{\tilde{\pi}}_{t_i})-\mathcal{Y}^{\lambda,\pi}_{t_i})^-.$$
Then, $|\eta|<C_2|\delta X_{t_i}|$ by Lipschitz continuity, and we obtain
\begin{equation}
\label{final1}
    \begin{aligned}
    (1+\eta_1+\eta_2)\delta Y^\lambda_{t_i}+\int^{t_{i+1}}_{t_i}(Z^\lambda_s-\mathcal{Z}^{\lambda,\pi}_{t_i})^\top\,dW_s\\=\delta Y^\lambda_{t_{i+1}}+\int^{t_{i+1}}_{t_i}\delta f_{t_i}(s)\,ds+ r_i+\lambda\Delta t\,\eta.\\
    \end{aligned}
\end{equation}
Thus, squaring both side of \eqref{final1} and taking expectation, and then using It\^o isometry yields
\begin{equation}
\label{final2}
    \begin{aligned}
        &\mathbb{E}\left[(1+\eta_1+\eta_2)^2(\delta Y^{\lambda}_{t_i})^2\right]+\mathbb{E}\left[\int^{t_{i+1}}_{t_i}\left|Z^\lambda_s-\mathcal{Z}^{\lambda,\pi}_{t_i}\right|^2\,ds\right]\\
        =&\mathbb{E}\left[\left(\delta Y^\lambda_{t_{i+1}}+\int^{t_{i+1}}_{t_i}\delta f_{t_i}(s)\,ds+ r_i+\eta\lambda\Delta t\right)^2\right].
    \end{aligned}
\end{equation}
Since $1+\eta_1+\eta_2>1$, applying Young's inequality to the right-hand side of \eqref{final2} gives us
\begin{equation}
    \label{final3}
    \begin{aligned}
        &\mathbb{E}\left|\delta Y^{\lambda}_{t_i}\right|^2+\mathbb{E}\left[\int^{t_{i+1}}_{t_i}\left|Z^\lambda_s-\mathcal{Z}^{\lambda,\pi}_{t_i}\right|^2\,ds\right]\\
        \le&(1+16K^2_z\Delta t)\mathbb{E}\left|\delta Y^\lambda_{t_{i+1}}\right|^2\\&+3\left(1+\frac{1}{16K^2_z\Delta t}\right)\mathbb{E}\left[\left(\int^{t_{i+1}}_{t_i}\delta f_{t_i}(s)\,ds\right)^2+|r_i|^2+\lambda^2\Delta t^2\eta^2\right]
    \end{aligned}
\end{equation}

Let us estimate $\mathbb{E}\left[\left(\int^{t_{i+1}}_{t_i}\delta f_{t_i}(s)\,ds\right)^2\right]$. By the continuity property of $f$ and Young's inequality, we have
\begin{equation}
\label{fterm}
\begin{aligned}&\mathbb{E}\left[\left(\int^{t_{i+1}}_{t_i}\delta f_{t_i}(s)\,ds\right)^2\right]\\
\le&\mathbb{E}\left[\left(\int^{t_{i+1}}_{t_i}\left|\delta f_{t_i}(s)\right|\,ds\right)^2\right]\\
\le& \mathbb{E}\Bigg[\Big(\int^{t_{i+1}}_{t_i}K_x|\delta X_{t_i}|+\int^{t_{i+1}}_{t_i} K_y\left|\delta Y^{\lambda}_{t_i}\right|+K_z\int^{t_{i+1}}_{t_i}|Z^\lambda_s-\mathcal{Z}^{\lambda,\pi}_{t_i}|\,ds\Big)^2\Bigg]\\
\le&4\mathbb{E}\Big|\int^{t_{i+1}}_{t_i}K_x|\delta X_{t_i}|\,ds\Big|^2+4\mathbb{E}\left|\int^{t_{i+1}}_{t_i} K_y\left|\delta Y^{\lambda}_{t_i}\right|\,ds\right|^2\\
&\qquad+2K^2_z\,\mathbb{E}\left|\int^{t_{i+1}}_{t_i}|Z^\lambda_s-\mathcal{Z}^{\lambda,\pi}_{t_i}|\,ds\right|^2
\end{aligned}
\end{equation}
We now bound each term on the right hand side of \eqref{fterm}. First, applying the Cauchy--Schwarz inequality, Young's inequality, and \eqref{xdifbound}, we have
\begin{equation}
    \label{fterm1}
    \begin{aligned}
        &4\mathbb{E}\Big|\int^{t_{i+1}}_{t_i}K_x|\delta X_{t_i}|\,ds\Big|^2\\
        \le&4K^2_x \mathbb{E}\left[\left(\int^{t_{i+1}}_{t_i}1\,ds\right)\left(\int^{t_{i+1}}_{t_i}|\delta X_{t_i}|^2\,ds\right)\right]\\
        \le&C_2 \Delta t\left(\left|\int^{t_{i+1}}_{t_i}\mathbb{E}|\delta X_{t_i}|^2\,ds\right|\right)\\
        \le&C_2\Delta t^2\Delta \tilde{t}\\
        \le&C_2\Delta t^3
    \end{aligned}
\end{equation}
For the second term, we find
\begin{equation}
    \label{fterm2}
    \begin{aligned}
        4\mathbb{E}\left|\int^{t_{i+1}}_{t_i} K_y\left|\delta Y^{\lambda}_{t_i}\right|\,ds\right|^2=4\mathbb{E}\left|K_y\left|\delta Y^{\lambda}_{t_i}\right|\int^{t_{i+1}}_{t_i} \,ds\right|^2=4K^2_y\Delta t^2\mathbb{E}\left|\delta Y^{\lambda}_{t_i}\right|^2.
    \end{aligned}
\end{equation}
We apply the Cauchy--Schwarz inequality to the third term to obtain
\begin{equation}
    \label{fterm3}
    \begin{aligned}
        2K^2_z\,\mathbb{E}\left|\int^{t_{i+1}}_{t_i}|Z^\lambda_s-\mathcal{Z}^{\lambda,\pi}_{t_i}|\,ds\right|^2
        \le&2K^2_z\,\mathbb{E}\left[\left(\int^{t_{i+1}}_{t_i}1^2\,ds\right)\left(\int^{t_{i+1}}_{t_i}|Z^\lambda_s-\mathcal{Z}^{\lambda,\pi}_{t_i}|^2\,ds\right)\right]\\
        =&2K^2_z\,\Delta t\,\mathbb{E}\left[\int^{t_{i+1}}_{t_i}|Z^\lambda_s-\mathcal{Z}^{\lambda,\pi}_{t_i}|^2\,ds\right].
    \end{aligned}
\end{equation}
Plugging \eqref{fterm1}, \eqref{fterm2}, and \eqref{fterm3} into \eqref{fterm} yields
\begin{equation}
    \label{ftermfinal}
    \begin{aligned}
        \mathbb{E}\left[\left(\int^{t_{i+1}}_{t_i}\delta f_{t_i}(s)\,ds\right)^2\right]
        \le& C_2\Delta t^3+4K^2_y\Delta t^2\mathbb{E}\left|\delta Y^{\lambda}_{t_i}\right|\\
        &+2K^2_z\,\Delta t\,\mathbb{E}\left[\int^{t_{i+1}}_{t_i}|Z^\lambda_s-\mathcal{Z}^{\lambda,\pi}_{t_i}|^2\,ds\right].
    \end{aligned}
\end{equation}

Thus,  \eqref{ftermfinal}, \eqref{final3}, and $\mathbb{E}[\eta^2]\le\mathbb{E}[\delta X^2_{t_i}]<C_2\Delta \tilde{t}$ gives us
\begin{equation}
    \label{final4}
    \begin{aligned}
    &\mathbb{E}\left|\delta Y^{\lambda}_{t_i}\right|^2+\frac{1}{2}\mathbb{E}\left[\int^{t_{i+1}}_{t_i}\left|Z^\lambda_s-\mathcal{Z}^{\lambda,\pi}_{t_i}\right|^2\,ds\right]\\
    \le&(1+16K^2_z\Delta t)\mathbb{E}\left|\delta Y^\lambda_{t_{i+1}}\right|^2\\
    &+C_2\Delta t^2
   +\left(\frac{K_{y}}{K_z}\right)^2\Delta t\,\mathbb{E}\left|\delta Y^\lambda_{t_{i}}\right|^2\\&\quad+\frac{C_2}{\Delta t}\,\mathbb{E}\left|r_i\right|^2+C_2\lambda^2\Delta t\,\Delta \tilde{t}.
    \end{aligned}
\end{equation}
Since $(1+16K^2_z\Delta t)^k<e^{16K^2_zT}$ for $k\le n$, induction on \eqref{final4} gives us
\begin{equation}
    \label{final5}
    \begin{aligned}
        &\mathbb{E}\left|\delta Y^{\lambda}_{t_i}\right|^2+\frac{1}{2}\mathbb{E}\left[\int^{t_{i+1}}_{t_i}|Z^\lambda_s-\mathcal{Z}^{\lambda,\pi}_{t_i}|^2\,ds\right]\\
        \le&e^{16K^2_zT}\,\mathbb{E}\left|\delta Y^\lambda_{T}\right|^2\\
    &+C_2\left(\Delta t+\frac{1}{\Delta t}\,\sum^{n-1}_{j=i}\mathbb{E}|r_j|^2\right)\\&
    +C_2\lambda^2\Delta \tilde{t}+2e^{16K^2_zT}\left(\frac{K_{y}}{K_z}\right)^2\,\sum^{n-1}_{j=i}\Delta t\,\mathbb{E}\left|\delta Y^\lambda_{t_{j+1}}\right|^2.
    \end{aligned}
\end{equation}
Therefore, applying Proposition~\ref{xyzbound} and Gr\"onwall's inequality to \eqref{final5} yields
\begin{equation}
    \label{final6}
    \begin{aligned}
        &\mathbb{E}\left|\delta Y^{\lambda}_{t_i}\right|^2+\mathbb{E}\left[\int^{t_{i+1}}_{t_i}|Z^\lambda_s-\mathcal{Z}^{\lambda,\pi}_{t_i}|^2\,ds\right]\\
        \le& C_2\left(\mathbb{E}\left|\delta Y^\lambda_{T}\right|^2+\Delta t+\frac{1}{\Delta t}\,\sum^{n-1}_{j=i}\mathbb{E}|r_j|^2\right)+C_2\lambda^2\Delta \tilde{t}
        \end{aligned}
\end{equation}

Let us estimate the term $\sum^{n-1}_{j=i}\mathbb{E}|r_j|^2$. Applying Young's inequality, continuity properties of $f$, Proposition~\ref{xyzbound}, and Corollary~\ref{apbcor}, we calculate
\begin{equation}
\label{rterm}
\begin{aligned}&\sum^{n-1}_{j=i}\mathbb{E}|r_j|^2\\
\le& \sum^{n-1}_{i=0}\mathbb{E}\Bigg|\int^{t_{i+1}}_{t_i}f\left(s,X_s,Y^\lambda_s,Z^\lambda_s\right)-f\left(t_i,X_{t_i},Y^\lambda_{t_i},Z^\lambda_s\right)\,ds\\
&+\int^{t_{i+1}}_{t_i}\lambda\left(\left(Y^\lambda_s-p_b(s,X_s)\right)^--\left(Y^\lambda_{t_i}-p_b(t_i,X_{t_i})\right)^-\right)\,ds\\
&+\int^{t_{i+1}}_{t_i}\lambda\left(\left(p_w(s,X_s)-Y^\lambda_s\right)^--\left(p_w(t_i,X_{t_i})-Y^\lambda_{t_i})\right)^-\right)\,ds\Bigg|^2\\
\le& 3\,\sum^{n-1}_{i=0}\mathbb{E}\Bigg|\int^{t_{i+1}}_{t_i}f\left(s,X_s,Y^\lambda_s,Z^\lambda_s\right)-f\left(t_i,X_{t_i},Y^\lambda_{t_i},Z^\lambda_s\right)\,ds\Bigg|^2\\
&+3\,\sum^{n-1}_{i=0}\mathbb{E}\left|\int^{t_{i+1}}_{t_i}\lambda\left(\left(Y^\lambda_s-p_b(s,X_s)\right)^--\left(Y^\lambda_{t_i}-p_b(t_i,X_{t_i})\right)^-\right)\,ds\right|^2\\
&+3\,\sum^{n-1}_{i=0}\mathbb{E}\left|\int^{t_{i+1}}_{t_i}\lambda\left(\left(p_w(s,X_s)-Y^\lambda_s\right)^--\left(p_w(t_i,X_{t_i})-Y^\lambda_{t_i})\right)^-\right)\,ds\right|^2\\
\end{aligned}
\end{equation}
Using \eqref{xdifbound2} and applying Young's inequality, the Cauchy--Schwarz inequality, and Corollary~\ref{apbcor}, we have
\begin{equation}
    \label{rterm1}
    \begin{aligned}
    &3\,\sum^{n-1}_{i=0}\mathbb{E}\Bigg|\int^{t_{i+1}}_{t_i}f\left(s,X_s,Y^\lambda_s,Z^\lambda_s\right)-f\left(t_i,X_{t_i},Y^\lambda_{t_i},Z^\lambda_s\right)\,ds\Bigg|^2\\
    \le&3\,\sum^{n-1}_{i=0}\mathbb{E}\Bigg|\int^{t_{i+1}}_{t_i}K_{f,\frac{1}{2}}\sqrt{s-t_i}+K_x|X_s-X_{t_i}|+K_y|Y^\lambda_s-Y^\lambda_{t_i}|\,ds\Bigg|^2\\
    \le& C_2\sum^{n-1}_{i=0}\mathbb{E}\Bigg|\int^{t_{i+1}}_{t_i}K_{f,\frac{1}{2}}\sqrt{s-t_i}+K_x|X_s-X_{t_i}|\,ds\Bigg|^2\\
    &+C_2\sum^{n-1}_{i=0}\mathbb{E}\Bigg|\int^{t_{i+1}}_{t_i}K_y|Y^\lambda_s-Y^\lambda_{t_i}|\,ds\Bigg|^2\\
    \le&C_2\sum^{n-1}_{i=0}\Delta t^3+C_2\Delta t\sum^{n-1}_{i=0}\int^{t_{i+1}}_{t_i}\mathbb{E}\left|Y^\lambda_s-Y^\lambda_{t_i}\right|^2\,ds\\
    \le& C_2 \Delta t^2.
    \end{aligned}
\end{equation}
Also, by Proposition~\ref{apbcor} and the Cauchy--Schwarz inequality, we find that
\begin{equation}
\label{rterm2}
\begin{aligned}
&3\,\sum^{n-1}_{i=0}\mathbb{E}\left|\int^{t_{i+1}}_{t_i}\lambda\left(\left(Y^\lambda_s-p_b(s,X_s)\right)^--\left(Y^\lambda_{t_i}-p_b(t_i,X_{t_i}))\right)^-\right)\,ds\right|^2\\
\le&C_2\lambda^2\,\sum^{n-1}_{i=0}\mathbb{E}\left|\Delta t\,\int^{t_{i+1}}_{t_i}\left|Y^\lambda_s-Y^\lambda_{t_i}\right|^2+K^2_{b,\frac{1}{2}}|s-t_i|+K^2_b|X_s-X_{t_i}|^2\,ds\right|\\
\le&C_2\lambda^2\Delta t\sum^{n-1}_{i=0}\left(\int^{t_{i+1}}_{t_i}\mathbb{E}\left|Y^\lambda_s-Y^\lambda_{t_i}\right|^2\,ds+\Delta t ^2+\int^{t_{i+1}}_{t_i}\mathbb{E}\left|X_s-X_{t_i}\right|^2\,ds\right)\\
\le&C_2\lambda^2\Delta t\sum^{n-1}_{i=0}\left(\int^{t_{i+1}}_{t_i}\mathbb{E}\left|Y^\lambda_s-Y^\lambda_{t_i}\right|^2\,ds+\Delta t^2\right)\\
\le& C_2\lambda^2\,\Delta t^2.
\end{aligned}
\end{equation}
By the same argument, we also have
\begin{equation}
    \label{rterm3}
    \begin{aligned}
    3\,\sum^{n-1}_{i=0}\mathbb{E}\left|\int^{t_{i+1}}_{t_i}\lambda\left(\left(p_w(s,X_s)-Y^\lambda_s\right)^--\left(p_w(t_i,X_{t_i})-Y^\lambda_{t_i})\right)^-\right)\,ds\right|^2\le C_2\lambda^2\,\Delta t^2.
    \end{aligned}
\end{equation}
Thus, plugging in \eqref{rterm1}, \eqref{rterm2}, and \eqref{rterm3} into \eqref{rterm}, we arrive at
\begin{equation}
    \label{rtermr}
    \sum^{n-1}_{j=i}\mathbb{E}|r_j|^2\le C_2\,\lambda^2\Delta t^2.
\end{equation}

Finally, plugging in \eqref{rtermr} into \eqref{final6}, we obtain 
$$\mathbb{E}|\delta Y^\lambda_{t_i}|^2<C_2(\mathbb{E}|\delta Y^\lambda_T|^2+\lambda^2\Delta t).$$
Plugging this into \eqref{final4}, we obtain
\begin{equation*}
    \begin{aligned}
    &\mathbb{E}\left|\delta Y^{\lambda}_{t_i}\right|^2+\frac{1}{2}\mathbb{E}\left[\int^{t_{i+1}}_{t_i}\left|Z^\lambda_s-\mathcal{Z}^{\lambda,\pi}_{t_i}\right|^2\,ds\right]\\
    \le&\mathbb{E}\left|\delta Y^\lambda_{t_{i+1}}\right|^2+C_2(\mathbb{E}|\delta Y^\lambda_T|^2+\lambda^2\Delta t)\Delta t\\
    &+C_2\lambda\Delta t^2+\left(\frac{K_{y}}{K_z}\right)^2\Delta t\,\mathbb{E}\left|\delta Y^\lambda_{t_{i}}\right|^2+\frac{C_2}{\Delta t}\,\mathbb{E}\left|r_i\right|^2+C_2\lambda^2\Delta t^2.
    \end{aligned}
\end{equation*}
Summing up , we arrive at
\begin{equation*}
    \begin{aligned}
    &\sum^{n-1}_{j=i}\left(\mathbb{E}\left|\delta Y^{\lambda}_{t_j}\right|^2+\frac{1}{2}\mathbb{E}\left[\int^{t_{j+1}}_{t_j}\left|Z^\lambda_s-\mathcal{Z}^{\lambda,\pi}_{t_j}\right|^2\,ds\right]\right)\\
    \le&\sum^{n-1}_{j=i}\mathbb{E}\left|\delta Y^\lambda_{t_{j+1}}\right|^2+C_2(\mathbb{E}|\delta Y^\lambda_T|^2+\lambda^2\Delta t)\\
    &+C_2\lambda\Delta t+\sum^{n-1}_{j=i}\left(\frac{K_{y}}{K_z}\right)^2\Delta t\,\mathbb{E}\left|\delta Y^\lambda_{t_{j}}\right|^2+\frac{C_2}{\Delta t}\,\sum^{n-1}_{j=i}\mathbb{E}\left|r_j\right|^2+C_2\lambda^2\Delta t.
    \end{aligned}
\end{equation*}
Plugging in \eqref{rterm} again, we find
\begin{equation*}
    \begin{aligned}
    &\mathbb{E}\left|\delta Y^{\lambda}_{t_i}\right|^2+\frac{1}{2}\sum^{n-1}_{i=j}\mathbb{E}\left[\int^{t_{j+1}}_{t_j}\left|Z^\lambda_s-\mathcal{Z}^{\lambda,\pi}_{t_j}\right|^2\,ds\right]\\
    \le&C_2\mathbb{E}\left|\delta Y^\lambda_T\right|^2\\
    &+C_2\lambda^2\Delta t+\sum^{n-1}_{j=i}\left(\frac{K_{y}}{K_z}\right)^2\Delta t\,\mathbb{E}\left|\delta Y^\lambda_{t_{j}}\right|^2+\frac{C_2}{\Delta t}\,\sum^{n-1}_{j=i}\mathbb{E}\left|r_j\right|^2+C_2\lambda^2\Delta t\\
    \le& C_2\left(\mathbb{E}|\delta Y^\lambda_T|^2+\lambda^2\Delta t+\sum^{n-1}_{j=i}\Delta t \mathbb{E}\left|\delta Y^{\lambda}_{t_i}\right|^2\right).
    \end{aligned}
\end{equation*}
Applying the discrete Gr\"onwall inequality again, we obtain the first desired inequality.

Furthermore, if \eqref{discterm} holds, then by the Lipschitz property of $\bar{g}$ and \eqref{xdifbound}, it is clear that
\begin{equation}
    \label{finalterm}
    \mathbb{E}\left|\delta Y^\lambda_{T}\right|^2\le C_2\Delta t.
\end{equation}
Plugging \eqref{finalterm} into the first inequality yields the second desired inequality.

\end{proof}

Combining the penalty error from Theorem~\ref{penaltyerror} and discretization error from Theorem~\ref{discerror}, we obtain the total error as follows.
\begin{corollary}
      Suppose Assumptions \ref{assumptions}, \ref{assumptionpde}, and \ref{assumption2} hold and let $(X^{\tilde{\pi}},\mathcal{Y}^{\lambda,\pi},\mathcal{Z}^{\lambda,\pi})$ satisfy \eqref{emx} and \eqref{scheme}. Let $(X,Y,Z,A,K)$ denote the solution to \eqref{BSDE} and choose $\lambda>2K_y$ and $\Delta t<1/(16K^2_z)$. Then, we have
         $$\left(\max_{i\in\{0,\cdots,n-1\}}\mathbb{E}\left|Y_{t_i}-\mathcal{Y}^{\lambda,\pi}_{t_i}\right|^2\right)^\frac{1}{2}
        \le C_2 \left((\mathbb{E}|g(X_T)-\mathcal{Y}^{\lambda,\pi}_T|^2)^\frac{1}{2}+\lambda\Delta t^\frac{1}{2}+\frac{1}{\lambda}\right)$$
        and 
        \begin{equation*}
        \begin{aligned}&\left(\sum^{n-1}_{i=0}\mathbb{E}\left[\int^{t_{i+1}}_{t_i}|Z_s-\mathcal{Z}^{\lambda,\pi}_{t_i}|^2\,ds\right]\right)^{\frac{1}{2}}
        \\\le& C_2 \left((\mathbb{E}|g(X_T)-\mathcal{Y}^{\lambda,\pi}_T|^2)^\frac{1}{2}+\lambda\Delta t^\frac{1}{2}+\frac{1}{\sqrt{\lambda}}\right)
        \end{aligned}
        \end{equation*}
Furthermore, if $(X^{\tilde{\pi}},\mathcal{Y}^{\lambda,\pi},\mathcal{Z}^{\lambda,\pi})$ also satisfies \eqref{discterm}, we have
         $$\left(\max_{i\in\{0,\cdots,n-1\}}\mathbb{E}\left|Y_{t_i}-\mathcal{Y}^{\lambda,\pi}_{t_i}\right|^2\right)^\frac{1}{2}
        \le C_2 \left(\lambda\Delta t^\frac{1}{2}+\frac{1}{\lambda}\right)$$
        and 
        \begin{equation*}
        \begin{aligned}&\left(\sum^{n-1}_{i=0}\mathbb{E}\left[\int^{t_{i+1}}_{t_i}|Z_s-\mathcal{Z}^{\lambda,\pi}_{t_i}|^2\,ds\right]\right)^{\frac{1}{2}}
        \le C_2 \left(\lambda\Delta t^\frac{1}{2}+\frac{1}{\sqrt{\lambda}}\right)
        \end{aligned}
        \end{equation*}
\end{corollary}

In particular, if $\lambda=\Delta t^{-\frac{1}{2}}$, we obtain the minimal error rate
$$\left(\max_{i\in\{0,\cdots,n-1\}}\mathbb{E}\left|Y_{t_i}-\mathcal{Y}^{\lambda,\pi}_{t_i}\right|^2\right)^{\frac{1}{2}}
        \le C_2 \Delta t^\frac{1}{4}$$
for the value process. This is suboptimal compared to the error rate of $O(\Delta t^{\frac{1}{2}})$ for conventional approximation schemes.
\subsection{Absolute error bound for Z-independent driver}
We consider the case in which the driver $f$ is independent of $Z$, and derive an absolute
error bound for the value process. In the nonsmooth setting, the Tanaka--Peskir decomposition
of the barriers produces surface local-time terms supported on the singular hypersurfaces
$\Gamma_1,\dots,\Gamma_q$. To control these terms uniformly, we impose the following
nondegeneracy condition in the normal direction to each $\Gamma_k$.

\begin{assumption}\label{assumption_surface}
For each $k\in\{1,\dots,q\}$, either there exist an open neighborhood $U_k$ of $\Gamma_k$
and a constant $\eta_k>0$ such that
\[
\nabla\varphi_k(x)^\top \sigma(t,x)\sigma(t,x)^\top \nabla\varphi_k(x)
\ge \eta_k |\nabla\varphi_k(x)|^2,
\qquad (t,x)\in [0,T]\times U_k,
\]
or, defining
\[
\tau_{\Gamma_k}:=\inf\{t\in[0,T]:X_t\in\Gamma_k\},
\]
we have
\[
P(\tau_{\Gamma_k}\le T)=0.
\]
\end{assumption}

Assumption~\ref{assumption_surface} is used only to control the expected surface
local-time terms arising in the Tanaka--Peskir formula. It is not imposed for mere finiteness
of the local times, but to obtain a uniform deterministic bound on their expectations.

As before, we work under Assumptions~\ref{assumptions}, \ref{assumptionpde},
\ref{assumption2}, and \ref{assumption_surface}. For each $k=1,\cdots,q$,
let $\delta_k$ be the signed distance to $\Gamma_k$ on a tubular neighborhood and define
\[
L_t^k:=L_t^0\big(\delta_k(X)\big),\qquad t\in[0,T],
\]
where $L^0(\cdot)$ denotes the symmetric semimartingale local time at $0$.
For $j\in\{b,w\}$ define
\[
A_{j,t}^\Gamma
:=\frac12\sum_{k=1}^q\int_0^t \big[\partial_{n_k}p_j\big](s,X_s)\,dL_s^k,
\qquad t\in[0,T],
\]
with $[\partial_{n_k}p_j](t,x):=\partial_{n_k}p_j(t,x+)-\partial_{n_k}p_j(t,x-)$.
 Under Assumption~\ref{assumption_surface}, it is standard that
\begin{equation}
    \label{lbound}
    \mathbb{E}[L^k_t]<C_2,\quad \forall t\in[0,T],\ \forall k=1,\cdots,q.
\end{equation}
See e.g., \cite{LGSM18}.

The following is an extended version of Theorem 2 of \cite{PWW:24}.
\begin{theorem}
\label{better}
    Let Assumptions \ref{assumptions},\ref{assumptionpde}, \ref{assumption2}, and \ref{assumption_surface} hold. Suppose $f(t,x,y,z)=f(t,x,y,0)$ for all $(t,x,y,z)\in[0,T]\times\mathbb{R}^d\times \mathbb{R}\times\mathbb{R}^d$.  We denote by $(X,Y^\lambda,Z^\lambda)$ the solution to \eqref{pBSDE} and choose $\lambda>2K_y$. Then, for the processes $(X^{\tilde{\pi}},\mathcal{Y}^{\lambda,\pi},\mathcal{Z}^{\lambda,\pi})$ satisfying \eqref{emx} and \eqref{scheme}, we have
     $$\max_{i\in\{0,\cdots,n\}}\mathbb{E}\left|Y^\lambda_{t_i}-\mathcal{Y}^{\lambda,\pi}_{t_i}\right|
        \le C_2 \left(\mathbb{E}|g(X_T)-\mathcal{Y}^{\lambda,\pi}_T|+\Delta t^\frac{1}{2}+\lambda\Delta \tilde{t}^\frac{1}{2}+\lambda\Delta t\right).$$
    Furthermore, if $(X^{\tilde{\pi}},\mathcal{Y}^{\lambda,\pi},\mathcal{Z}^{\lambda,\pi})$  satisfies \eqref{discterm}, we obtain
         $$\max_{i\in\{0,\cdots,n\}}\mathbb{E}\left|Y^\lambda_{t_i}-\mathcal{Y}^{\lambda,\pi}_{t_i}\right|
        \le C_2 \left(\Delta t^\frac{1}{2}+\lambda\Delta \tilde{t}^\frac{1}{2}+\lambda\Delta t\right).$$
\end{theorem}
\begin{proof}
We use the same notations as in the proof of Theorem~\ref{discerror}. Taking conditional expectation of both sides of \eqref{final1} and then taking absolute values, we obtain

\begin{equation}
\label{ffinal2}
(1+\eta_1+\eta_2)\left|\delta Y^{\lambda}_{t_i}\right| =\left|\mathbb{E}\left[\delta Y^\lambda_{t_{i+1}}+\delta f_{t_i}\Delta t+ r_i+\lambda\Delta t\,\eta\Big| \mathcal{F}_{t_i}\right]\right|.
\end{equation}
Since $1+\eta_1+\eta_2>1$, we find that
\begin{equation}
    \label{ffinal3}
\mathbb{E}\left|\delta Y^{\lambda}_{t_i}\right|
        \le \mathbb{E}\left|\delta Y^\lambda_{t_{i+1}}\right|+\mathbb{E}\left|\delta f_{t_i}\right|\Delta t+\mathbb{E}\left|\mathbb{E}[r_i\mid\mathcal{F}_{t_i}]|\right|+C_2\lambda\Delta t\Delta \tilde{t}^\frac{1}{2}.
\end{equation}

We now estimate the third term $\mathbb{E}\left|\mathbb{E}[r_i\mid\mathcal{F}_{t_i}]|\right|$ in the right-hand side of the above inequality. Observe that
\begin{equation}
\label{rtermnew}
\begin{aligned}&\mathbb{E}[|\mathbb{E}[r_j\mid\mathcal{F}_{t_i}]|]\\
\le& \mathbb{E}\Bigg|\int^{t_{i+1}}_{t_i}f\left(s,X_s,Y^\lambda_s,0\right)-f\left(t_i,X_{t_i},Y^\lambda_{t_i},0\right)\,ds\Bigg|\\
&+\mathbb{E}\Bigg|\int^{t_{i+1}}_{t_i}\lambda\,\mathbb{E}\left[\left(Y^\lambda_s-p_b(s,X_s)\right)^--\left(Y^\lambda_{t_i}-p_b(t_i,X_{t_i})\right)^-\Big|\mathcal{F}_{t_i}\right]\,ds\Bigg|\\
&+\mathbb{E}\Bigg|\int^{t_{i+1}}_{t_i}\lambda\mathbb{E}\left[\left(p_w(s,X_s)-Y^\lambda_s\right)^--\left(p_w(t_i,X_{t_i})-Y^\lambda_{t_i})\right)^-\Big|\mathcal{F}_{t_i}\right]\,ds\Bigg|.
\end{aligned}
\end{equation}
Then, we have
\begin{equation}
    \label{rterm1new}
    \begin{aligned}
    &\mathbb{E}\Bigg|\int^{t_{i+1}}_{t_i}f\left(s,X_s,Y^\lambda_s,0\right)-f\left(t_i,X_{t_i},Y^\lambda_{t_i},0\right)\,ds\Bigg|\\
    \le& C_2\int^{t_{i+1}}_{t_i}\left(\sqrt{s-t_i}+\mathbb{E}[|X_s-X_{t_i}|^2]^\frac{1}{2}+\mathbb{E}|Y^\lambda_s-Y^\lambda_{t_i}|\right)\,ds\\
    \le& C_2\left(\Delta t^\frac{3}{2}+\int^{t_{i+1}}_{t_i}\mathbb{E}|Y^\lambda_s-Y^\lambda_{t_i}|\,ds\right).
    \end{aligned}
\end{equation}
Also, applying the change-of-variable formula with local time on surfaces
\citep{Peskir2007Surfaces}, we calculate that for $t_i\le s\le t_{i+1}$,
\begin{equation*}
\begin{aligned}
&\mathbb{E}\left[\left(Y^\lambda_s-p_b(s,X_s)\right)^-
-\left(Y^\lambda_{t_i}-p_b(t_i,X_{t_i})\right)^-\Big|\mathcal{F}_{t_i}\right]\\
=&\mathbb{E}\Bigg[
-\int_{t_i}^s \mathbf 1_{\{Y^\lambda_u-p_b(u,X_u)\le 0\}}\,d\left(Y^\lambda_u-p_b(u,X_u)\right)
+\frac12\left(L^b_s-L^b_{t_i}\right)\Big|\mathcal{F}_{t_i}\Bigg],
\end{aligned}
\end{equation*}
where $L^b:=L^0(Y^\lambda_\cdot-p_b(\cdot,X_\cdot))$. Moreover, by Peskir's formula we have
\[
dp_b(u,X_u)=\mathcal{L}p_b(u,X_u)\,du+(\nabla_xp_b\,\sigma)(u,X_u)\,dW_u+dA^\Gamma_{b,u}.
\]
Thus, using $dY^\lambda_u=-f_\lambda(u,X_u,Y^\lambda_u,Z^\lambda_u)\,du+Z^\lambda_u\,dW_u$, we obtain
\begin{equation*}
\begin{aligned}
&\mathbb{E}\left[\left(Y^\lambda_s-p_b(s,X_s)\right)^-
-\left(Y^\lambda_{t_i}-p_b(t_i,X_{t_i})\right)^-\Big|\mathcal{F}_{t_i}\right]\\
=&\mathbb{E}\Bigg[\int_{t_i}^s \mathbf 1_{\{Y^\lambda_u-p_b(u,X_u)\le 0\}}
\Big(f_\lambda(u,X_u,Y^\lambda_u,Z^\lambda_u)+\mathcal L p_b(u,X_u)\Big)\,du\\
&\qquad\quad
+\int_{t_i}^s \mathbf 1_{\{Y^\lambda_u-p_b(u,X_u)\le 0\}}\,dA^\Gamma_{b,u}
\\&\qquad\qquad-\int_{t_i}^s \mathbf 1_{\{Y^\lambda_u-p_b(u,X_u)\le 0\}}
\Big(Z^\lambda_u-(\nabla_x p_b\,\sigma)(u,X_u)\Big)\,dW_u\\
&\qquad\quad\qquad
+\frac12\left(L^b_s-L^b_{t_i}\right)\Bigg|\mathcal{F}_{t_i}\Bigg].
\end{aligned}
\end{equation*}
Taking conditional expectations eliminates the stochastic integral, hence
\begin{equation*}
\begin{aligned}
&\mathbb{E}\left[\left(Y^\lambda_s-p_b(s,X_s)\right)^-
-\left(Y^\lambda_{t_i}-p_b(t_i,X_{t_i})\right)^-\Big|\mathcal{F}_{t_i}\right]\\
=&\mathbb{E}\Bigg[\int_{t_i}^s \mathbf 1_{\{Y^\lambda_u-p_b(u,X_u)\le 0\}}
\Big(f_\lambda(u,X_u,Y^\lambda_u,Z^\lambda_u)+\mathcal L p_b(u,X_u)\Big)\,du
\\&\quad+\int_{t_i}^s \mathbf 1_{\{Y^\lambda_u-p_b(u,X_u)\le 0\}}\,dA^\Gamma_{b,u}+\frac12\left(L^b_s-L^b_{t_i}\right)\Bigg|\mathcal{F}_{t_i}\Bigg].
\end{aligned}
\end{equation*}
Thus, by Proposition~\ref{xyzbound}, continuity properties of $f$, polynomial growth of $\mathcal Lp_b$,
Corollary~\ref{stoapb}, and \eqref{xbound}, we have
\begin{equation}
\label{rterm2new}
\begin{aligned}
&\mathbb{E}\left|\int^{t_{i+1}}_{t_i}\mathbb{E}\left[\left(Y^\lambda_s-p_b(s,X_s)\right)^-
-\left(Y^\lambda_{t_i}-p_b(t_i,X_{t_i})\right)^-\Big|\mathcal{F}_{t_i}\right]\,ds\right|\\
\le&\mathbb{E}\left|\int^{t_{i+1}}_{t_i}\int^s_{t_i}\left|f\left(u,X_u,Y^\lambda_u,0\right)\right|\,du\,ds\right|\\
&+\mathbb{E}\left|\int^{t_{i+1}}_{t_i}\int^s_{t_i}\left|\lambda\left((Y^\lambda_u-p_b(u,X_u))^--(p_w(u,X_u)-Y^\lambda_u)^-\right)\right|\,du\,ds\right|\\
&+\mathbb{E}\left|\int^{t_{i+1}}_{t_i}\int^s_{t_i}\left|\mathcal{L}p_b(u,X_u)\right|\,du\,ds\right|\\
&+\mathbb{E}\left|\int^{t_{i+1}}_{t_i}\int^s_{t_i} d\big|A^\Gamma_{b}\big|_u\,ds\right|
+\mathbb{E}\left|\int^{t_{i+1}}_{t_i}\frac{1}{2}\left(L^b_s-L^b_{t_i}\right)\,ds\right|\\
\le& C_2\Big(\Delta t^2+\Delta t\,\mathbb{E}\left[L^b_{t_{i+1}}-L^b_{t_i}\right]
+\Delta t\,\sum_{k=1}^q\mathbb{E}\left[L^k_{t_{i+1}}-L^k_{t_i}\right]\Big).
\end{aligned}
\end{equation}
Here $|A^\Gamma_b|$ denotes the total variation process of $A^\Gamma_b$, and we used that
$d|A^\Gamma_b|_u\le C_2\sum_{k=1}^q dL_u^k$ by the bounded jump condition (Assumption~\ref{assumption2}(ii)) on $[\partial_{n_k}p_b]$.
By the same argument, we also have
\begin{equation}
\label{rterm3new}
\begin{aligned}
&\mathbb{E}\left|\int^{t_{i+1}}_{t_i}\mathbb{E}\left[\left(p_w(s,X_s)-Y^\lambda_s\right)^-
-\left(p_w(t_i,X_{t_i})-Y^\lambda_{t_i}\right)^-\Big|\mathcal{F}_{t_i}\right]\,ds\right|\\
\le&\; C_2\left(\Delta t^2+\Delta t\,\mathbb{E}\left[L^w_{t_{i+1}}-L^w_{t_i}\right]
+\Delta t\,\sum_{k=1}^q\mathbb{E}\left[L^k_{t_{i+1}}-L^k_{t_i}\right]\right),
\end{aligned}
\end{equation}
where $L^w_\cdot:=L^0\big(p_w(\cdot,X_\cdot)-Y^\lambda_\cdot\big)$ and $|A^\Gamma_w|$ denotes the total variation process of
 $A^\Gamma_w$.
Thus, plugging in \eqref{rterm1new}, \eqref{rterm2new}, and \eqref{rterm3new} into \eqref{rtermnew}, we obtain
\begin{equation}
\label{rtermn}
\begin{aligned}
\mathbb{E}\Big[\big|\mathbb{E}[r_j\mid\mathcal{F}_{t_i}]\big|\Big]
\le C_2\Big(&\Delta t^{\frac{3}{2}}+\lambda\Delta t^2
+\lambda\Delta t\,\mathbb{E}\big[L^{b,w}_{t_{i+1}}-L^{b,w}_{t_i}\big]\\
&+\lambda\Delta t\,\sum_{k=1}^q\mathbb{E}\left[L^k_{t_{i+1}}-L^k_{t_i}\right]\Big),
\end{aligned}
\end{equation}
where $L^{b,w}:=L^b+L^w$.

We now estimate
$\mathbb{E}\left|\delta Y^{\lambda}_{t_i}\right|$ in \eqref{ffinal3}.
We obtain
\begin{equation}
\label{fff}
\begin{aligned}
\mathbb{E}\left|\delta Y^{\lambda}_{t_i}\right|
\le& \mathbb{E}\left|\delta Y^\lambda_{t_{i+1}}\right|
+C_2\Big(\Delta t^{\frac{3}{2}}+\lambda\Delta t^2
+\lambda\Delta t\,\mathbb{E}\big[L^{b,w}_{t_{i+1}}-L^{b,w}_{t_i}\big]\Big)\\
&+C_2\lambda\Delta t\sum_{k=1}^q\mathbb{E}\left[L^k_{t_{i+1}}-L^k_{t_i}\right]
+C_2\Delta t\,\mathbb{E}\left|\delta Y^\lambda_{t_i}\right|
+C_2\lambda\Delta t\,\Delta \tilde{t}^{\frac{1}{2}}
\end{aligned}
\end{equation}
from \eqref{rtermn}  and 
\begin{equation}
\label{t1}
    \begin{aligned}
        \mathbb{E}|\delta f_{t_i}|&\le \mathbb{E}\left|K_x|X_{t_i}-X^{\tilde{\pi}}_{t_i}|+K_y\left|\delta Y^\lambda_{t_i}\right|\right|\\
        &\le C_2\left([\mathbb{E}[|X_{t_i}-X^{\tilde{\pi}}_{t_i}|^2]^\frac{1}{2}+ \mathbb{E}\left|\delta Y^\lambda_{t_i}\right|\right)\\
        &\le C_2\left(\Delta t^\frac{1}{2}+ \mathbb{E}\left|\delta Y^\lambda_{t_i}\right|\right).
    \end{aligned}
\end{equation}
which is achieved by the Cauchy--Schwarz inequality, Proposition~\ref{xyzbound}, and \eqref{xdifbound}.
Summing up the inequalities \eqref{fff}, we obtain
\begin{equation}
\label{done1}
\begin{aligned}
\sum_{j=i}^{n-1}\mathbb{E}\left|\delta Y^{\lambda}_{t_j}\right|
\le&\sum_{j=i}^{n-1}\mathbb{E}\left|\delta Y^\lambda_{t_{j+1}}\right|
+C_2\sum_{j=i}^{n-1}\Big(\Delta t^{\frac{3}{2}}+\lambda\Delta t^2+\lambda\Delta t\,\Delta \tilde{t}^{\frac{1}{2}}\Big)\\
&+C_2\lambda\Delta t\sum_{j=i}^{n-1}\mathbb{E}\big[L^{b,w}_{t_{j+1}}-L^{b,w}_{t_j}\big]
+C_2\lambda\Delta t\sum_{j=i}^{n-1}\sum_{k=1}^q\mathbb{E}\big[L^k_{t_{j+1}}-L^k_{t_j}\big]\\
&+C_2\sum_{j=i}^{n-1}\Delta t\,\mathbb{E}\left|\delta Y^\lambda_{t_j}\right|\\
\le&\sum_{j=i}^{n-1}\mathbb{E}\left|\delta Y^\lambda_{t_{j+1}}\right|
+C_2\Big(\Delta t^{\frac12}+\lambda\Delta t+\lambda\Delta \tilde{t}^{\frac12}\Big)\\
&+C_2\lambda\Delta t\,\mathbb{E}\big[L^{b,w}_T-L^{b,w}_0\big]
+C_2\lambda\Delta t\sum_{k=1}^q\mathbb{E}\left[L^k_T-L^k_0\right]\\
&+C_2\sum_{j=i}^{n-1}\Delta t\,\mathbb{E}\left|\delta Y^\lambda_{t_j}\right|.
\end{aligned}
\end{equation}

Finally, we derive the first desired inequality. 
Note that, by Peskir's change-of-variable formula,
\begin{equation*}
\begin{aligned}
\mathbb{E}\big[L^b_T-L^b_0\big]
=&\,\mathbb{E}\Big[|Y^\lambda_T-p_b(T,X_T)|-|Y^\lambda_0-p_b(0,X_0)|\Big]\\
&+\mathbb{E}\Bigg[\int_0^T \mathrm{sgn}\big(Y^\lambda_u-p_b(u,X_u)\big)
\Big(f_\lambda(u,X_u,Y^\lambda_u,Z^\lambda_u)+\mathcal L p_b(u,X_u)\Big)\,du\Bigg]\\
&+\mathbb{E}\Bigg[\int_0^T \mathrm{sgn}\big(Y^\lambda_u-p_b(u,X_u)\big)\,dA^\Gamma_{b,u}\Bigg].
\end{aligned}
\end{equation*}
Thus,
\begin{equation*}
\begin{aligned}
\mathbb{E}\big[L^b_T-L^b_0\big]
\le&\;\mathbb{E}|Y^\lambda_T|+\mathbb{E}|p_b(T,X_T)|
+\mathbb{E}|Y^\lambda_0|+\mathbb{E}|p_b(0,X_0)|\\
&+\int_0^T\mathbb{E}\big|f(u,X_u,Y^\lambda_u,0)\big|\,du
+\int_0^T\mathbb{E}\big|\mathcal L p_b(u,X_u)\big|\,du\\
&+\int_0^T\mathbb{E}\big|\lambda(Y^\lambda_u-p_b(u,X_u))^-\big|\,du
+\int_0^T\mathbb{E}\big|\lambda(p_w(u,X_u)-Y^\lambda_u)^-\big|\,du\\
&+\mathbb{E}\big[|A^\Gamma_b|_T-|A^\Gamma_b|_0\big]\\
\le&\; C_2\left(1+\sum_{k=1}^q\mathbb{E}\left[L^k_{T}-L^k_{0}\right]\right),
\end{aligned}
\end{equation*}
where we used Proposition~\ref{xyzbound}, and Corollary~\ref{stoapb}.
Similarly, $\mathbb{E}\big[L^w_T-L^w_0\big]\le C_2\left(1+\sum_{k=1}^q\mathbb{E}\left[L^k_{T}-L^k_{0}\right]\right)$, and hence
\begin{equation}
\label{done2}
\mathbb{E}\big[L^{b,w}_T-L^{b,w}_0\big]\le C_2\left(1+\sum_{k=1}^q\mathbb{E}\left[L^k_{T}-L^k_{0}\right]\right).
\end{equation}
Moreover, it follows from \eqref{lbound} that 
\begin{equation}\label{done3}\sum_{k=1}^q\mathbb{E}\left[L^k_{T}-L^k_{0}\right]\le C_2.
\end{equation}
Therefore, putting in \eqref{done2} and \eqref{done3} into \eqref{done1}, we arrive at
\begin{equation*}
\begin{aligned}
\mathbb{E}\left|\delta Y^{\lambda}_{t_i}\right|
        \le&\mathbb{E}\left|\delta Y^\lambda_T\right|+ C_2\left(\Delta t^\frac{1}{2}+\lambda\Delta \tilde{t}^\frac{1}{2}+\lambda\Delta t+\sum^{n-1}_{j=i}\Delta t\mathbb{E}\left|\delta Y^\lambda_{t_j}\right|\right).
        \end{aligned}
\end{equation*}
Applying the discrete Gr\"onwall inequality yields the first desired inequality.

Also, if \eqref{discterm} holds, then the Lipschitz property of $g$ and \eqref{xdifbound} implies $\mathbb{E}|\delta Y^\lambda_T|<C_2\Delta t^\frac{1}{2}$, yielding the second desired inequality.
\end{proof}

Combining Theorem~\ref{better} and Theorem~\ref{penaltyerror}, we obtain the total absolute error bound as follows.
\begin{corollary}
\label{cor:final}
      Suppose Assumptions \ref{assumptions}, \ref{assumptionpde}, \ref{assumption2}, and \ref{assumption_surface} hold and let $(X^{\tilde{\pi}},\mathcal{Y}^{\lambda,\pi},\mathcal{Z}^{\lambda,\pi})$ satisfy \eqref{emx} and \eqref{discretization}. Let $(X,Y,Z,A,K)$ denote the solution to \eqref{BSDE} and choose $\lambda>2K_y$. If $f(t,x,y,z)=f(t,x,y,0)$ for all $(t,x,y,z)\in[0,T]\times\mathbb{R}^d\times\mathbb{R}\times\mathbb{R}^d$, then we have
         \begin{equation*}\begin{aligned}
         &\max_{i\in\{0,\cdots,n-1\}}\mathbb{E}\left|Y_{t_i}-\mathcal{Y}^{\lambda,\pi}_{t_i}\right|
        \\\le& C_2 \Big((\mathbb{E}|g(X_T)-\mathcal{Y}^{\lambda,\pi}_T|+\Delta t^\frac{1}{2}+\lambda\Delta \tilde{t}^\frac{1}{2}+\lambda\Delta t+\frac{1}{\lambda}\Big).\end{aligned}\end{equation*}
Furthermore, if $(X^{\tilde{\pi}},\mathcal{Y}^{\lambda,\pi},\mathcal{Z}^{\lambda,\pi})$ also satisfies \eqref{discterm}, we have
         $$\max_{i\in\{0,\cdots,n-1\}}\mathbb{E}\left|Y_{t_i}-\mathcal{Y}^{\lambda,\pi}_{t_i}\right|
        \le C_2 \left(\Delta t^\frac{1}2+\lambda\Delta\tilde{t}^\frac{1}{2}+\lambda\Delta t+\frac{1}{\lambda}\right).$$
\end{corollary}

It is easy to see that the optimal error rate of $O(\Delta t^\frac{1}{2})$ is achieved with $\lambda=\Delta t^{-\frac{1}{2}}$ and $\Delta \tilde{t}=\Delta t/m$ for $m>(\lambda^2/C_2),m\in\mathbb{N}$.

\begin{remark}\label{rem:opt_rate}
In the bound of Corollary~\ref{cor:final}, the leading terms are
\[
\Delta t^{\frac12} \;+\; \lambda\,\Delta t \;+\; \lambda\,\Delta \tilde t^{\frac12} \;+\; \frac{1}{\lambda}.
\]
To obtain the target rate $O(\Delta t^{1/2})$, the auxiliary time step $\Delta \tilde t$ must be chosen so that
\[
\lambda\,\Delta \tilde t^{\frac12} = O(\Delta t^{\frac12}),
\qquad\text{equivalently}\qquad
\Delta \tilde t = O\!\left(\frac{\Delta t}{\lambda^{2}}\right).
\]
Under this coupling, choosing $\lambda \asymp \Delta t^{-1/2}$ yields
\[
\max_{i\in\{0,\dots,n-1\}}\mathbb{E}\bigl|Y_{t_i}-\mathcal{Y}^{\lambda,\pi}_{t_i}\bigr|
\;=\; O(\Delta t^{1/2}).
\]
\end{remark}

\section{Experimental results}

\paragraph{Problem setup and DRBSDE.}
We work under the risk-neutral Black--Scholes model
\[
dX_t = rX_t\,dt + \sigma X_t\,dW_t,\qquad X_0=24,\qquad t\in[0,T],\qquad T=1,
\]
with parameters \(r=0.02\) and \(\sigma=0.2\). We consider a one-dimensional Dynkin game (Israeli/game put) in which the holder may exercise to receive the put payoff, whereas the writer may cancel by paying an additional time-dependent premium. In the DRBSDE formulation, we solve for the adapted processes \((Y,Z,A,K)\) such that
\[
Y_t
=
g(X_T)
+\int_t^T (-rY_s)\,ds
+\bigl(A_T-A_t\bigr)
-\bigl(K_T-K_t\bigr)
-\int_t^T Z_s\,dW_s,
\qquad t\in[0,T],
\]
where \(A\) and \(K\) are nondecreasing processes with \(A_0=K_0=0\), together with the obstacle condition
\[
p_b(t,X_t)\le Y_t \le p_w(t,X_t),\qquad t\in[0,T],
\]
and the Skorokhod reflection conditions
\[
\int_0^T \bigl(Y_t-p_b(t,X_t)\bigr)\,dA_t=0,
\qquad
\int_0^T \bigl(p_w(t,X_t)-Y_t\bigr)\,dK_t=0.
\]
Here
\[
g(x)=(K-x)^+,\qquad
p_b(t,x)=(K-x)^+,\qquad
p_w(t,x)=p_b(t,x)+\delta(T-t),
\]
with strike \(K=25\) and \(\delta=1\). The quantity of interest is the initial value \(Y_0\).

\paragraph{Experimental procedure and parameter sweeps.}
We compare our penalized implicit discretization, implemented with the LSMC regression (\cite{LS01}) for conditional expectations against a penalty-free reference computed by a CRR binomial Dynkin-game recursion (\cite{CRR79}). The CRR baseline uses a recombining tree with a large number of steps, where \(N_{\mathrm{tree}}=10{,}000\), thereby yielding the reference value
\[
Y_0^{\mathrm{ref}}=1.626263795927.
\]

For our method, we simulate \(N=60{,}000\) sample paths of \(S\), repeat the experiment over three random seeds, and report the mean estimate of \(Y_0\) along with the absolute and relative errors
\[
\mathrm{abs\_err}=|Y_0-Y_0^{\mathrm{ref}}|,\qquad 
\mathrm{rel\_err}=\frac{|Y_0-Y_0^{\mathrm{ref}}|}{|Y_0^{\mathrm{ref}}|}.
\]

At each backward time step \(t_i\), the conditional expectation is approximated using a cross-sectional least-squares regression on the current state variable \(S_i\). More precisely, writing
\[
x_i=\frac{\log S_i-\overline{\log S_i}}{\mathrm{std}(\log S_i)},
\]
we use a linear spline (ReLU) basis of the form
\[
1,\quad x_i,\quad (x_i-k_1)^+,\ldots,(x_i-k_M)^+,\quad \frac{(K-S_i)^+}{K},\quad \mathbf 1_{\{S_i<K\}},
\]
where the knots \(k_1,\dots,k_M\) are chosen as empirical quantiles of \(x_i\). In the following experiments, we take \(M=30\) knots uniformly over the empirical quantile range from \(5\%\) to \(95\%\). The regression coefficients are computed using ridge least squares with ridge parameter \(10^{-8}\).

We perform two parameter sweeps.

\begin{enumerate}
\item \textbf{Time-step sweep.}
We vary the number of backward time steps
\[
n\in\{25,50,100,200,400,800,1600,3200\},\qquad \Delta t = T/n,
\]
and choose the penalty parameter according to the theoretically motivated scaling
\[
\lambda=\lambda_0 n^{1/2},\qquad \lambda_0=2000.
\]
Since Corollary~\ref{cor:final} predicts an \(O(\Delta t^{1/2})\) convergence rate in the \(Z\)-independent case under the coupling \(\lambda\asymp \Delta t^{-1/2}\), this experiment tests whether the observed error behaves like \(n^{-1/2}\). In addition to the measured errors, we fit a curve of the form \(Cn^{-1/2}\) to the observed relative errors.

\item \textbf{Penalty sweep at fixed \(n\).}
Fixing \(n=200\), we vary the penalty parameter according to
\[
\lambda=\lambda_{\mathrm{scale}}\,n^{a},\qquad \lambda_{\mathrm{scale}}=20,
\]
for several values of \(a\). The second experiment is intended to assess the finite-resolution sensitivity of the method to the penalty level and to determine whether the worst case balancing effect is visible for the example at hand.
\end{enumerate}
The forward process is sampled using the Euler--Maruyama scheme under the Black--Scholes model with $m=n$.

\paragraph{Results.}
Table~\ref{tab:grid_sweep} and Figure~\ref{fig:grid_sweep} summarize the time step sweep with
\[
\lambda=2000\,n^{1/2}.
\]
The observed relative error decreases steadily as \(n\) increases. Moreover, the graph closely follows the fitted profile
\[
Cn^{-1/2},\qquad C=3.93\times 10^{-1},
\]
over the first four grid levels, and for the finer grid levels the measured error falls slightly below the fitted curve. This behavior is consistent with the theoretical \(O(n^{-1/2})\) prediction.

Table~\ref{tab:penalty_sweep} and Figure~\ref{fig:penalty_sweep} show the penalty sweep at fixed \(n=200\). Over the tested range, the relative error decreases monotonically as \(\lambda\) increases. This behavior does not contradict the theory. Rather, it indicates that the experiment remains in a regime in which stronger penalization continues to improve the approximation, whereas the worst-case balancing effect is not visible for the particular case.

\begin{figure}
    \centering
    \includegraphics[width=0.8\linewidth]{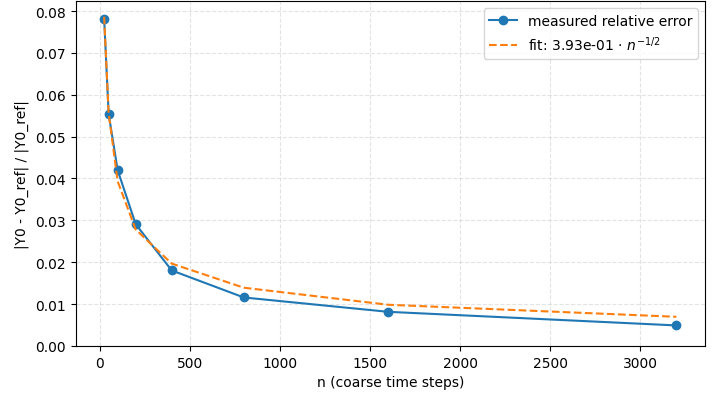}
    \caption{Relative error as a function of the coarse grid size \(n\) under the theoretically motivated choice \(\lambda=2000\,n^{1/2}\). The dashed curve represents the fitted profile \(Cn^{-1/2}\) with \(C=3.93\times 10^{-1}\). The measured error closely follows the fitted \(n^{-1/2}\) trend for the first four grid levels and then falls slightly below it, in agreement with the theoretical prediction. The plot is generated from the values reported in Table~\ref{tab:grid_sweep}.}
    \label{fig:grid_sweep}
\end{figure}

\begin{table}
\centering
\begin{tabular}{r r r r r r}
\hline
\(n,m\) & \(\Delta t\) & \(\lambda\) & \(Y_0\) & abs.\ error & rel.\ error \\
\hline
25   & \(4.000\times 10^{-2}\) & \(1.000\times 10^{4}\)   & 1.7534482612 & \(1.272\times 10^{-1}\) & \(7.821\times 10^{-2}\) \\
50   & \(2.000\times 10^{-2}\) & \(1.414\times 10^{4}\)   & 1.7163806797 & \(9.012\times 10^{-2}\) & \(5.541\times 10^{-2}\) \\
100  & \(1.000\times 10^{-2}\) & \(2.000\times 10^{4}\)   & 1.6945257566 & \(6.826\times 10^{-2}\) & \(4.197\times 10^{-2}\) \\
200  & \(5.000\times 10^{-3}\) & \(2.828\times 10^{4}\)   & 1.6735441304 & \(4.728\times 10^{-2}\) & \(2.907\times 10^{-2}\) \\
400  & \(2.500\times 10^{-3}\) & \(4.000\times 10^{4}\)   & 1.6555788985 & \(2.932\times 10^{-2}\) & \(1.803\times 10^{-2}\) \\
800  & \(1.250\times 10^{-3}\) & \(5.657\times 10^{4}\)   & 1.6450752175 & \(1.881\times 10^{-2}\) & \(1.157\times 10^{-2}\) \\
1600 & \(6.250\times 10^{-4}\) & \(8.000\times 10^{4}\)   & 1.6395064975 & \(1.324\times 10^{-2}\) & \(8.143\times 10^{-3}\) \\
3200 & \(3.125\times 10^{-4}\) & \(1.131\times 10^{5}\)   & 1.6342116681 & \(7.948\times 10^{-3}\) & \(4.887\times 10^{-3}\) \\
\hline
\end{tabular}
\caption{Grid-refinement sweep with \(\lambda=2000\,n^{1/2}\). The table reports the penalized-LSMC estimate \(Y_0\) and the corresponding absolute and relative errors with respect to the CRR reference value \(Y_0^{\mathrm{ref}}=1.626263795927\).}
\label{tab:grid_sweep}
\end{table}
\begin{figure}
    \centering
    \includegraphics[width=0.8\linewidth]{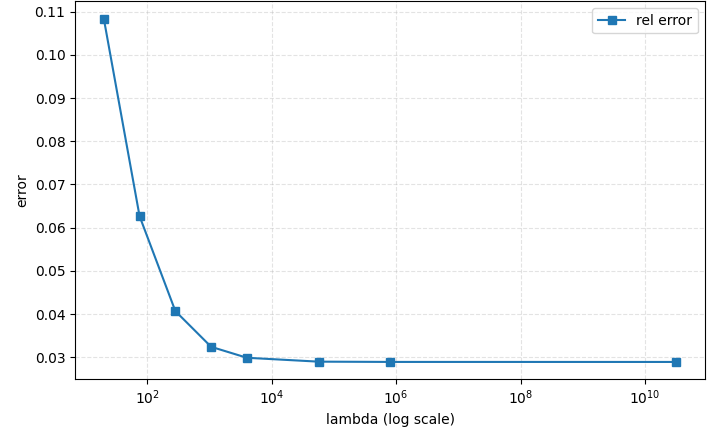}
    \caption{Relative error as a function of the penalty parameter \(\lambda\), with \(\lambda=20\,n^a\) for several values of \(a\). Over the tested range, the relative error decreases monotonically as \(\lambda\) increases. This indicates that the experiment remains in a regime where increasing penalization improves the approximation, while the asymptotic balancing effect predicted by the theory is not yet visibly resolved. The plot is generated from the values reported in Table~\ref{tab:penalty_sweep}.}
    \label{fig:penalty_sweep}
\end{figure}

\begin{table}
\centering
\begin{tabular}{r r r r r}
\hline
\(a\) & \(\lambda\) & \(Y_0\) & abs.\ error & rel.\ error \\
\hline
0    & \(2.000\times 10^{1}\)   & 1.8025543928 & \(1.763\times 10^{-1}\) & \(1.084\times 10^{-1}\) \\
0.25 & \(7.52121\times 10^{1}\) & 1.7281770784 & \(1.019\times 10^{-1}\) & \(6.267\times 10^{-2}\) \\
0.5  & \(2.82843\times 10^{2}\) & 1.6924566711 & \(6.619\times 10^{-2}\) & \(4.070\times 10^{-2}\) \\
0.75 & \(1.06366\times 10^{3}\) & 1.6790144395 & \(5.275\times 10^{-2}\) & \(3.244\times 10^{-2}\) \\
1    & \(4.000\times 10^{3}\)   & 1.6748887935 & \(4.862\times 10^{-2}\) & \(2.990\times 10^{-2}\) \\
1.5  & \(5.65685\times 10^{4}\) & 1.6734319322 & \(4.717\times 10^{-2}\) & \(2.900\times 10^{-2}\) \\
2    & \(8.000\times 10^{5}\)   & 1.6733274770 & \(4.706\times 10^{-2}\) & \(2.894\times 10^{-2}\) \\
4    & \(3.200\times 10^{10}\)  & 1.6733195215 & \(4.706\times 10^{-2}\) & \(2.893\times 10^{-2}\) \\
\hline
\end{tabular}
\caption{Penalty-parameter sweep at fixed \(n=200\), with \(\lambda=20\,n^a\). The table reports the penalized-LSMC estimate \(Y_0\), and the corresponding absolute and relative errors with respect to the CRR reference value \(Y_0^{\mathrm{ref}}=1.626263795927\).}
\label{tab:penalty_sweep}
\end{table}

\section{Conclusion}
\label{sec:con}

We analyzed a penalty approximation of a decoupled Markovian doubly reflected BSDE in the computational regime where penalization is coupled with time discretization. The central difficulty in the two-barrier setting is that the error produced by evaluating the obstacles along an approximated forward process is multiplied by the penalty parameter \(\lambda\). In contrast to the single-barrier case, no simple transformation removes this amplification for both obstacles simultaneously. This motivates the two-grid construction, in which the forward process is resolved separately from the backward discretization.

Under additional structural assumptions motivated by typical financial barriers, we sharpened the penalization error and obtained a uniform \(O(\lambda^{-1})\) bound for the value process. Combining this with the fully discrete scheme yields explicit error bounds in \((\Delta t,\tilde{\Delta t},\lambda)\) and transparent parameter coupling rules. In particular, in the \(Z\)-independent case, choosing
\[
\lambda\asymp \Delta t^{-1/2},
\qquad
\tilde{\Delta t}=O\!\left(\frac{\Delta t}{\lambda^2}\right)
\]
recovers the target \(O(\Delta t^{1/2})\) rate for the value process.

The numerical experiments support this picture. In the grid-refinement experiment, performed under the theoretically motivated scaling \(\lambda=2000\,n^{1/2}\), the observed relative error follows the predicted \(n^{-1/2}\) behavior closely. In the penalty sweep at a fixed resolution, the error continues to decrease as \(\lambda\) increases over the tested range, suggesting that the balancing effect is not present for our example.

However, several directions remain unexplored. Theoretically, it would be natural to extend the analysis beyond the decoupled Markovian setting, and investigate whether comparable quantitative results can be obtained for more general obstacle geometries or less regular coefficients. On the numerical side, it would be interesting to combine the present error analysis with a more flexible regression or deep-learning-based solvers in higher dimensions.

\paragraph{Acknowledgments}
 Hyungbin Park was supported by the National Research Foundation of Korea (NRF) grants funded by the Ministry of Science and ICT (Nos. 2021R1C1C1011675, 2022R1A5A6000840, RS-2026-25488333). Financial support from the Institute for Research in Finance and Economics of Seoul National University is gratefully acknowledged.

\bibliographystyle{spmpsci}
\bibliography{TGPDRBSDE}

\end{document}